\documentclass[10pt,letterpaper]{article}
\makeatletter
\usepackage{pdfsync}%pdf反向搜索用到的包
\usepackage{cite}
\usepackage{amscd}
\usepackage{amsmath}
\usepackage{latexsym}
\usepackage{amsfonts}
\usepackage{amssymb}
\usepackage{amsthm}
\usepackage{graphicx}
\usepackage{indentfirst}
\usepackage{enumerate}
\usepackage{amstext}
\usepackage[dvipdfm,
            pdfstartview=FitH,
            CJKbookmarks=true,
            bookmarksnumbered=true,
            bookmarksopen=true,
            colorlinks=true, %注释掉此项则交叉引用为彩色边框(将colorlinks和pdfborder 同时注释掉)
            pdfborder=001,   %注释掉此项则交叉引用为彩色边框
            citecolor=magenta,% magenta , cyan
            linkcolor=blue,
            ]{hyperref}       % hyperref 宏包通常要求放在导言区的最后!!!
\newtheorem{thm}{\textbf Theorem}[section]
\newtheorem{lem}{\textbf Lemma}[section]

\newtheorem{prop}{\textbf Proposition}[section]

\numberwithin{equation}{section}

\newcommand{\be}{\begin{eqnarray}}
\newcommand{\ee}{\end{eqnarray}}

\newcommand{\bes}{\begin{eqnarray*}}
\newcommand{\ees}{\end{eqnarray*}}

\begin{document}
\begin{titlepage}
\title{\bf Existence and stability of periodic solution to the 3D Ginzburg-Landau equation
in weighted Sobolev spaces }
%\author{Boling Guo,
%   Guoquan Qin\thanks{Corresponding Author: G. Qin} \\
% \small     Institute of Applied Physics and Computational Mathematics,\\
%\small  China Academy of Engineering Physics,  Beijing,  100088,  P. R. China.\\
%      \small Graduate School of China Academy of Engineering
%       Physics,
%        Beijing, 100088,  P. R.  China.\\
%\small        (gbl@iapcm.ac.cn, qinguoquan16@gscaep.ac.cn)
%          }

%\title{Low mach number limit of strong solutions to 3-D full Navier-Stokes equations with Dirichlet boundary condition\hspace{-4mm}}
\author{Boling Guo$^{~a}$,  \quad Guoquan Qin$~^{b,*}$
\\[10pt]
\small {$^a $ Institute of Applied Physics and Computational Mathematics,
China Academy of Engineering Physics,}\\
\small {   Beijing,  100088,  P. R. China}\\[5pt]
\small {$^b $ Graduate School of China Academy of Engineering Physics,}\\
\small {  Beijing,  100088,  P. R. China}\\[5pt]
}
%\address{Department of Mathematics and Statistics, Indian Institute of Technology Kanpur, \\Kanpur, Uttar Pradesh, India\\}
\footnotetext
{*~Corresponding author.

~~~~~E-mail addresses: gbl@iapcm.ac.cn(B. Guo), 690650952@qq.com(G. Qin).}

%\author{
%Boling Guo$^1$, Guoquan Qin$^2$
%\address{
%\small  $^1$\\ Institute of Applied Physics and Computational Mathematics,
%        China Academy of Engineering Physics,  Beijing,  100088,  P. R. China.\\
%\small  $^2$Graduate School of China Academy of Engineering
%      Physics, Beijing, 100088,  P. R.  China.\\
%\small Email:gbl@iapcm.ac.cn, qinguoquan16@gscaep.ac.cn
%}
%}

\date{}
\end{titlepage}
\maketitle
\begin{abstract}
We prove the existence of time periodic solution to the
3D Ginzburg-Landau equation
in weighted Sobolev spaces.
We consider the cubic Ginzburg-Landau equation
with an external force $g$ satisfying the oddness condition $g(-x,t)=-g(x,t)$.
The existence of the periodic solution is proved for  small
time-periodic external force.
The stability of the time periodic solution
is also considered.

\end{abstract}

~~\noindent{ \textbf{Key words}: Periodic solution; Ginzburg-Landau equation;
Weighted Sobolev space
 }

%\vspace
%%\newpage

%%%%%%%%%%%%%%%%%%%%%%%%%%%%%%%%%%%%%%%%%%%%%%%%%%%%%%%%%%%%%%%%%%%%%%%%%%%%%%%%%%%%%%%%%%%%%%%%%%%%%%%%%%%%%%%%%%%%%%%%%%%%

\section{Introduction}
\setcounter{equation}{0}
In this paper, we are concerned with  the following  Ginzburg-Landau equation
\begin{eqnarray}\label{gl}
\begin{cases}
u_{t}-(1+i)\Delta u=|u|^{2}u+g, \ \ x\in \mathbb{R}^{3}, \ \ t\in \mathbb{R},\\
u(t+T)=u(t),\ \ x\in \mathbb{R}^{3},\ \ t\in \mathbb{R}
\end{cases}
\end{eqnarray}
where $u=u(t,x)$ is a complex function,
and $g=g(x, t)$ is a given
external force.

Equation (\ref{gl}) without the external force $g$
is a special form of the following
generalized derivative  Ginzburg-Landau equation
\begin{eqnarray}\label{glg}
u_{t}-(a+i\alpha)\Delta u+(b+i\beta)|u|^{\sigma}u+
|u|^{\delta}(\vec{\lambda}\cdot\nabla u)+|u|^{\delta-2}u^{2}(\vec{\mu}\cdot\nabla \bar{u})=0.
\end{eqnarray}

The Ginzburg-Landau equation
 has a long history in the research  of phase transitions and superconductivity.
  In fluid mechanical systems,  it is a  generic amplitude equation near the
onset of instabilities that lead to chaotic dynamics.
 It is also of  particularly interest due to it's certain connection with the nonlinear
Schr\"{o}dinger equation.

This paper is devoted to the study of the existence of a periodic in time
solution to equation (\ref{gl}) with the external force $g$ satisfying
\begin{eqnarray}\label{Tg}
\begin{cases}
g(x, t+T)=g(x,t), \quad x\in \mathbb{R}^{3}, t\in \mathbb{R},\text{ $T>0$ is a fixed constant, }\\
g(-x, t)=-g(x,t),\quad x\in \mathbb{R}^{3}, t\in \mathbb{R}.
\end{cases}
\end{eqnarray}

Time periodic problems for the Ginzburg-Landau equation have
been extensively investigated in the last decades.
Guo-Jing-Lu \cite{GuoJingLu1998i,GuoJingLu1998ii}
studied  the following cubic-quintic Ginzburg-Landau
equation
\begin{equation}\label{5gl}
  u_{t}=c_{0}u+(c_{0}+i\epsilon c_{1})u_{xx}
  -(\frac{c_{0}}{2}+i\epsilon c_{2})|u|^{2}u-(\frac{c_{0}}{2}+i\epsilon c_{3})|u|^{4}u.
\end{equation}
The existence and properties of the equilibria  of the perturbed system and
 unperturbed system to (\ref{5gl}) are investigated in \cite{GuoJingLu1998i}.
In \cite{GuoJingLu1998ii}, spatial quasiperiodic solutions to
 (\ref{5gl}) are proved to disappear due to
the perturbation and several types of heteroclinic orbits are proposed by
mathematical and numerical analysis.
The following generalized Ginzburg-Landau equation
\begin{equation}\label{6gl}
  u_{t}=\alpha_{0}u+\alpha_{1}u_{xx}
  +\alpha_{2}|u|^{2}u+\alpha_{3}|u|^{2}u_{x}+\alpha_{4}u^{2}\bar{u}_{x}
  +\alpha_{5}|u|^{4}u+f
\end{equation}
is researched by Guo-Yuan \cite{GuoYuan2000,GuoYuan2001}.
The almost periodic solution
to (\ref{6gl}) is established in Guo-Yuan \cite{GuoYuan2001}
provided  that the external force $f$
is almost periodic in time and the coefficients satisfy some conditions.
The one dimension case was investigated in \cite{GuoYuan2000} and periodic
solutions in Sobolev spaces was established under some conditions on the
coefficients and $f.$
Employing a priori estimate and the Leray-Schauder fixed point theorem,
Guo-Jiang \cite{GuoJiang2001} showed the existence and uniqueness of the time-periodic
solution to the Ginzburg-Landau  equation coupled with the BBM equation.
Guo-Yuan \cite{GuoYuan2002} considered equation (\ref{glg})
plus a damping term $u$ and an external force $f$ and established the time
periodic solution in the setting that  $f$ is  periodic in
time variable and $u$ is periodic in spatial variable.
Equation (\ref{glg}) in 3D plus a damping term $u$ and an external force $f$
but without derivative in the nonlinear terms was investigated
by Li-Guo \cite{LiGuo2004} and periodic solution in Sobolev spaces
was obtained with the help of Faedo-Schauder fix point theorem and the standard compactness
arguments.
For more results concerning the periodic solution to the Ginzburg-Landau equation,
one can refer to \cite{Almog1999,DoelmanGardnerJones1995,PorubovVelarde1999,SirovichNewton1986,Yang1988,Zhan2008}.
For results with respect to other aspects to the Ginzburg-Landau equation,
one can refer to
\cite{DuanHolmes1994,DuanHolmesTiti1992,GuoWang1995,GaoWang2009,HanWangGuo2012,LiGuo2000,
WangGuoZhao2004}.

This paper aims to  look for time periodic solution to equation (\ref{gl})
with external force $g$ satisfying (\ref{Tg})
in weighted Sobolev spaces.
The idea  is due to Kagei-Tsuda \cite{KageiTsuda2015}
and Tsuda \cite{Tsuda2016i,Tsuda2016ii,Tsuda2016iii}.

Before going further,  we
define the following functions
\begin{equation*}
\hat{\chi}_{1}(\xi)=\left\{\begin{array}{ll}{1} & {\left(|\xi| \leq r_{1}\right),} \\ {0} & {\left(|\xi| \geq r_{\infty}\right)}\end{array}\right.
\end{equation*}
and
\begin{equation*}
\hat{\chi}_{\infty}(\xi)=1-\hat{\chi}_{1}(\xi)
\end{equation*}
that satisfy
\begin{equation*}
\hat{\chi}_{j}(\xi) \in C^{\infty}\left(\mathbb{R}^{n}\right)
  \quad and \quad \quad 0 \leq \hat{\chi}_{j} \leq 1 \quad for \quad j=1, \infty,
\end{equation*}
where  $0<r_{1}<r_{\infty}$ will be fixed so that $(\ref{r})$ holds for $|\xi| \leq r_{\infty}$.

Based on the above two smooth functions,  we define
operators $P_{1}$ and $P_{\infty}$ on $L^{2}$ by
\begin{equation*}
P_{j} f=\mathcal{F}^{-1} \hat{\chi}_{j} \mathcal{F}[f] \quad\left(f \in L^{2}, j=1, \infty\right),
\end{equation*}
where
$\mathcal{F}[f]$
and $\hat{f}$  denote the Fourier transform of $f$  that is defined by
\begin{equation*}
\hat{f}(\xi)=\mathcal{F}[f](\xi)=\int_{\mathbb{R}^{3}} f(x) e^{-i x \cdot \xi} \mbox{d} x \quad\left(\xi \in \mathbb{R}^{3}\right)
\end{equation*}
and
$\mathcal{F}^{-1}[f]$ is the
 inverse Fourier transform of $f$ defined by
 \begin{equation*}
\mathcal{F}^{-1}[f](x)=(2 \pi)^{-3} \int_{\mathbb{R}^{3}} f(\xi) e^{i \xi \cdot x} \mbox{d} \xi \quad\left(x \in \mathbb{R}^{3}\right).
\end{equation*}

Note that $P_{1}$ and $P_{\infty}$
 decompose a function into its low and high frequency parts, respectively.

Set $F(u,g) \equiv |u|^{2}u+g$.

Applying  $P_{1}$ and $P_{\infty}$ to equation (\ref{gl}), respectively,
one finds
\begin{eqnarray}\label{gl01}
\begin{cases}
\partial_{t}u_{1}+A u_{1}= F_{1}(u,g),\\
\partial_{t}u_{\infty}+A u_{\infty}= F_{\infty}(u,g),
\end{cases}
\end{eqnarray}
where $A=-(1+i)\Delta,$ $u=u_{1}+u_{\infty}$,
$u_{j}=P_{j}u$ and $F_{j}=P_{j}F$ for $j=1,\infty.$
If we can find a periodic solution $u_{1}(t)=u_{1}(t+T)$
and $u_{\infty}(t)=u_{\infty}(t+T)$, then we see that
$u=u_{1}+u_{\infty}$ is a periodic solution satisfying equation (\ref{gl}).

To find a periodic solution to system (\ref{gl01}),
we first construct its local  solution
satisfying $u_{j}(0)=u_{j}(T)$, then we expand the local solution
to a global periodic one.
The reason we applying $P_{j}$ to  (\ref{gl01}) is that
$|1-e^{(1+i)T\theta}|\geq c\theta$ when $\theta>0$
is sufficiently small and this admits us to obtain
some decay in frequency space when $|\xi|$ is small enough.

To state our main results, let us first introduce some notations.

We use $L^{p}(1\leq p\leq\infty)$ and $H^{k}$ to denote the usual Lebesgue space and
Sobolev space, respectively.
The weighted Lebesgue space $L_{1}^{p}$ is defined by
$$L_{1}^{p}=\{
f\in L^{p}: \|f\|_{L_{1}^{p}}=\|(1+|x|)f\|_{L^{p}}<\infty.
\}$$
The weighted Sobolev space $H_{1}^{k}$ is defined by
$$H_{1}^{k}=\{
f\in H^{k}: \|f\|_{H_{1}^{k}}=\|(1+|x|)f\|_{H^{k}}<\infty.
\}$$

We define solution space to (\ref{gl01}) in which  the low frequency part will live
by
\begin{eqnarray*}\label{spu1}
X(a,b)=\left\{u_{1}: \operatorname{supp}\hat{u}_{1}\subset \{|\xi|\leq r_{\infty}\},
u_{1}(-x)=-u_{1}(x), \|u_{1}\|_{X(a,b)}<\infty
\right\},
\end{eqnarray*}
where
\begin{eqnarray*}\label{spu1}
\|u_{1}\|_{X(a,b)}=\|u_{1}\|_{H^{1}(a,b; L^{2})}
+\|x\nabla u_{1}\|_{H^{1}(a,b; L^{2})}
+\|\partial_{t} u_{1}\|_{L^{2}(a,b; L_{1}^{2})}.
\end{eqnarray*}
The solution space to (\ref{gl01}) in which  the high frequency part will live
is defined by
\begin{eqnarray*}\label{spu2}
Y(a,b)=\left\{u_{\infty}: \operatorname{supp}\hat{u}_{\infty}\subset \{|\xi|\geq r_{1}\},
u_{\infty}(-x)=-u_{\infty}(x), \|u_{\infty}\|_{Y(a,b)}<\infty
\right\},
\end{eqnarray*}
where
\begin{eqnarray*}\label{spu1}
\|u_{\infty}\|_{Y(a,b)}=\|u_{\infty}\|_{C([a,b]; H^{2}_{1})}
+\|u_{\infty}\|_{L^{2}(a,b; H^{3}_{1})}
+\|u_{\infty}\|_{H^{1}(a,b; H^{1}_{1})}.
\end{eqnarray*}
We define space $Z(a,b)\equiv X(a,b)\times Y(a,b)$
equipped with the norm
\begin{eqnarray*}\label{spu3}
\|\{u_{1},u_{\infty}\}\|_{Z(a,b)}=\|u_{1}\|_{X(a,b)}+\|u_{\infty}\|_{Y(a,b)}.
\end{eqnarray*}
We denote  by $C_{per}(\mathbb{R}; X)$ the set of all $T-$periodic
continuous functions with values in $X$ equipped with the norm $\|\cdot\|_{C([0, T]; X)}$
and by $L_{per}^{2}(\mathbb{R}; X)$ the set of all $T-$periodic
locally square integrable functions with values in $X$ equipped with the norm
$\|\cdot\|_{L^{2}(0, T; X)}$, and so on.

Let $B_{X(a, b)}(r)$ be the unit closed ball of $Z(a,b)$ centered at 0 with radius $r$,
that is,
\begin{eqnarray*}
B_{X(a, b)}(r)=\left\{
\{u_{1},u_{\infty}\}\in X(a,b);
\|\{u_{1},u_{\infty}\}\|_{X(a,b)}\leq r
\right\}.
\end{eqnarray*}

Set
\begin{eqnarray*}\label{spg}
[g]\equiv \|g\|_{L^{2}(0, T; L_{1}^{1})}+\|g\|_{L^{2}(0, T; H_{1}^{1})}.
\end{eqnarray*}

Our first  result  concerning the uniquely existence of the periodic solution
to the (\ref{gl}) in weighted Sobolev space
reads
\begin{thm}\label{thm1}
Let the  external force
 $g\in L_{per}^{2}(\mathbb{R}; H^{1}_{1}(\mathbb{R}^{3})\cap L^{1}_{1}(\mathbb{R}^{3}))$
 and (\ref{Tg}) be satisfied by $g$.
There exists constants $\delta_{0}>0$ and $C_{0}>0$ such that if $[g]\leq \delta_{0}$,
then equation (\ref{gl}) admits a periodic solution $u_{per}$ with period $T$
that satisfies $\{u_{1}, u_{\infty}\}\in Z_{per}(\mathbb{R})$,
where $u_{1}=P_{1}u$ and $u_{\infty}=P_{\infty}u$ and there holds
\begin{equation*}
  \|\{u_{1}, u_{\infty}\}\|_{Z(0, T)}\leq C_{0}[g].
\end{equation*}
Define $W=\{u: \{P_{1}u, P_{\infty}u\}\in Z_{per}(\mathbb{R}),
\|\{P_{1}u, P_{\infty}u\}\|_{Z(0, T)}\leq C_{0}\delta_{0}
\}.$
Then the  uniqueness of periodic solution of (\ref{gl}) holds in $W.$
\end{thm}

After obtaining the time periodic solution to (\ref{gl}),
a natural question is to ask whether this periodic solution is stable
under small perturbation.

Suppose that $v_{per}$ is the time periodic solution  to (\ref{gl})
 established in Theorem \ref{thm1}.

Then the perturbation $w=v-v_{per}$ satisfy
\begin{equation}\label{gldiff}
  \partial_{t}w-(1+i)\Delta w=2|v_{per}|^{2}w+v_{per}^{2}\bar{w}+|w|^{2}w+|w|^{2}v_{per}.
\end{equation}
We investigate the initial value problem of  equation (\ref{gldiff}) with initial condition $w(0)=w_{0}.$

The second result concerning with the stability of the time periodic solution
obtained in Theorem \ref{thm1} reads
\begin{thm}\label{thm2}
Let the  external force
 $g\in L_{per}^{2}(\mathbb{R}; H^{1}_{1}(\mathbb{R}^{3})\cap L^{1}_{1}(\mathbb{R}^{3}))$
 and (\ref{Tg}) be satisfied by $g$.
Suppose that $v_{per}$ is the time periodic solution established in Theorem \ref{thm1}
and assume $w_{0}\in H^{2}\cap L^{1}.$
If there is a  constant $\eta$ small enough such that
\begin{equation*}
  \|w_{0}\|_{H^{1}\cap L^{1}}+[g]\leq \eta,
\end{equation*}
then equation (\ref{gldiff}) admits a unique global solution $w\in C([0, T]; H^{2})$
satisfying
\begin{equation*}
  \|\nabla^{l}w(t)\|_{L^{2}}\leq C(1+t)^{-\frac{3}{4}-\frac{l}{2}}
\end{equation*}
for $t\in [0, \infty]$ and $l=0,1.$
\end{thm}

The rest of the paper is organized as follows.
In section 2, we collect some useful Lemmas.
In section 3, we deal with the low frequency  part of (\ref{gl01}).
With help of weighted energy method, the high frequency  part of (\ref{gl01}) is treated in section 4.
Section 5 devotes to estimate the nonlinear and nonhomogeneous terms of (\ref{gl01}).
We finally give the proof of Theorems \ref{thm1} and
\ref{thm2} in Sections 6 and 7, respectively.

\section{Preliminaries}

The following useful Lemmas can be found in \cite{KageiTsuda2015}.

\begin{lem}\label{lem201}
We have
\begin{equation*}
\|f\|_{L^{\infty}(\mathbb{R}^{3})} \leq C\|\nabla f\|_{H^{1}(\mathbb{R}^{3})}
\end{equation*}
for $f \in H^{2}(\mathbb{R}^{3})$.
\end{lem}

\begin{lem}\label{lem403}
For  $k$ being a nonnegative integer and  $2 \leq p \leq \infty,$
suppose that a function $f$  satisfies $\operatorname{supp}\hat{f}\subset \{|\xi|\leq r_{\infty}\}.$
Let  $f_{1}=P_{1}f.$
Then it finds
\begin{eqnarray*}
\left\|\nabla^{k} f_{1}\right\|_{L^{2}}+\left\|f_{1}\right\|_{L^{p}} \leq C\left\|f_{1}\right\|_{L^{2}} \quad\left(f \in L^{2}\right),\\
\left\|f_{1}\right\|_{H_{1}^{k}} \leq C\left\|f_{1}\right\|_{L_{1}^{2}}\quad\left(f \in L^{2} \cap L_{1}^{2}\right),\\
\left\|\nabla f_{1}\right\|_{H_{1}^{k}} \leq C(\left\|f_{1}\right\|_{L^{2}}
+\|x\nabla f_{1}\|_{L^{2}}) \quad\left(f \in L^{2}, x\nabla f\in L^{2} \right),\\
\left\|f_{1}\right\|_{L_{1}^{2}}+\left\|f_{1}\right\|_{L^{2}}
+\|x\nabla f_{1}\|_{L^{2}} \leq C\left\|f_{1}\right\|_{L_{1}^{1}} \quad\left(f \in L^{2} \cap L_{1}^{1}\right).
\end{eqnarray*}
\end{lem}

\begin{lem}\label{lem404}
(i)\ \ For  $k$ being a nonnegative integer, one has
 $\|P_{\infty}f\|_{H^{k}}\leq C\|f\|_{H^{k}}.$

(ii)\ \
Suppose that a function $f_{\infty}$  satisfies $\operatorname{supp}\hat{f}_{\infty}\subset \{|\xi|\geq r_{1}\}$ and $f_{\infty} \in H^{1},$
then it finds
\begin{equation*}
\left\|f_{\infty}\right\|_{L^{2}} \leq C\left\|\nabla f_{\infty}\right\|_{L^{2}}.
\end{equation*}
\end{lem}

\begin{lem}\label{lem405}
For  $\chi$ being a Schwartz function on $\mathbb{R}^{n}$, there holds
\begin{equation*}
\left\||x| (\chi*f)\right\|_{L^{2}} \leq
C\left\|f\right\|_{L^{2}}+C\left\|x f\right\|_{L^{2}}.
\end{equation*}
\end{lem}

\begin{lem}\label{lem406}
Let $f_{\infty}$ be satisfying $\operatorname{supp}\hat{f}_{\infty}\subset \{|\xi|\geq r_{1}\}$
and $\|f_{\infty}\|_{H^{1}_{1}}<\infty.$
Then there is a  constant $C>0$ independent of $f_{\infty}$ such
that
\begin{equation*}
\frac{r_{1}^{2}}{2}\left\||x| f_{\infty}\right\|_{L^{2}}^{2} \leq
\left\||x| \nabla f_{\infty}\right\|_{L^{2}}^{2}+C\left\|f_{\infty}\right\|_{L^{2}}^{2}.
\end{equation*}
\end{lem}

The following Hardy  inequality can be found in \cite{Evans1998,Okita2014}.
\begin{lem}\label{lemhardy}
\begin{equation}
\left\|\frac{u}{|x|}\right\|_{L^{2}(\mathbb{R}^{3})}\leq C\|\nabla u\|_{L^{2}(\mathbb{R}^{3})}
\end{equation}
for $u\in H^{1}(\mathbb{R}^{3}).$
\end{lem}

\section{Estimates of the low frequency part}
This section studies the following  low frequency  equation
\begin{eqnarray}\label{gllow}
\begin{cases}
\partial_{t}u_{1}+A u_{1}=F_{1},\\
u_{1}|_{t=0}=u_{10},\\
u_{1}(0)=u_{1}(T),
\end{cases}
\end{eqnarray}
where $F_{1}$ is assumed to satisfy $\operatorname{supp}\hat{F}_{1}\subset\{ |\xi|\leq r_{\infty}\}$.

Formally, the solution of equation (\ref{gllow}) can be written  as
\begin{eqnarray}\label{formulalow}
u_{1}(t)=\mbox{e}^{-tA}u_{10}+\int_{0}^{t}\mbox{e}^{-(t-s)A}F_{1}(s)\mbox{d}s.
\end{eqnarray}

Let $u_{1}(0)=u_{1}(T)$ in (\ref{formulalow}), we have
\begin{eqnarray}\label{low1}
u_{10}=\mbox{e}^{-TA}u_{10}+\int_{0}^{T}\mbox{e}^{-(T-s)A}F_{1}(s)\mbox{d}s,
\end{eqnarray}
which yields
\begin{eqnarray}\label{low2}
u_{10}=(1-\mbox{e}^{-TA})^{-1}\int_{0}^{T}\mbox{e}^{-(T-s)A}F_{1}(s)\mbox{d}s
\end{eqnarray}
provided that $(1-\mbox{e}^{-TA})^{-1}$ exists in some sense.

Substituting (\ref{low2}) into (\ref{formulalow}) leads to
\begin{eqnarray}\label{formulalow1}
u_{1}(t)=\mbox{e}^{-tA}(1-\mbox{e}^{-TA})^{-1}\int_{0}^{T}\mbox{e}^{-(T-s)A}F_{1}(s)\mbox{d}s
+\int_{0}^{t}\mbox{e}^{-(t-s)A}F_{1}(s)\mbox{d}s.
\end{eqnarray}

This leads to  the following Proposition
concerning the solvability of (\ref{gllow}).
\begin{prop}\label{myprop1}
If $F_{1}$ satisfies the following three conditions
\begin{enumerate}
  \item  $F_{1}(-x)=-F_{1}(x);$
  \item  $\operatorname{supp}\hat{F}_{1}\subset\{ |\xi|\leq r_{\infty}\};$
  \item  $F_{1}\in L^{2}(0, T; L^{2}\cap L^{1}_{1}),$
\end{enumerate}
then  $u_{1}(t)$ defined by (\ref{formulalow1}) is  a solution of (\ref{gllow})
in $X(0, T)$
and there holds
\begin{equation}\label{lowin1}
  \|u_{1}\|_{X(0, T)}\leq C\|F_{1}\|_{L^{2}(0, T;L_{1}^{1})}.
\end{equation}
\end{prop}

To prove Proposition \ref{myprop1}, we first investigate the properties of $e^{-tA}.$
We have the following Proposition describing the properties of $e^{-tA}.$

\begin{prop}\label{myprop2}
We have the following assertions:

\begin{enumerate}
  \item
   If $u_{1}$ satisfies $\operatorname{supp}\hat{u}_{1}\subset \{|\xi|\leq r_{\infty}\}$
  and $\|u_{1}\|_{L^{2}}+\|x \nabla u_{1}\|_{L^{2}}<\infty$,
  then $\operatorname{supp}\widehat{e^{-tA}u_{1}}\subset \{|\xi|\leq r_{\infty}\}$ and
  one finds
  \begin{eqnarray}
   \|e^{-tA}u_{1}\|_{L^{2}}+
   \|\partial_{t}e^{-tA}u_{1}\|_{L^{2}}\leq C\|u_{1}\|_{L^{2}},\label{in1}\\
   \|x \nabla e^{-tA}u_{1}\|_{L^{2}}
    +\|x \partial_{t}\nabla e^{-tA}u_{1}\|_{L^{2}}\leq C(\|u_{1}\|_{L^{2}}+\|x \nabla u_{1}\|_{L^{2}}), \label{in2}\\
    \|\partial_{t}e^{-tA}u_{1}\|_{L_{1}^{2}}\leq C(\|u_{1}\|_{L^{2}}+\|x \nabla u_{1}\|_{L^{2}}),\label{in3}
  \end{eqnarray}
  where in the above three inequalities, $t\in[0, T^{\prime}]$
  with $T^{\prime}>0$ any given number and $C$ depends on $T^{\prime}.$
  \item If $\operatorname{supp}\hat{F}_{1}\subset \{|\xi|\leq r_{\infty}\}$
  and $\|F_{1}\|_{L^{2}(0, T; L^{2})}+\|x \nabla F_{1}\|_{L^{2}(0, T; L^{2})}<\infty$,
  then $\operatorname{supp}\mathcal{F}[\int_{0}^{t}\mbox{e}^{-(t-s)A}F_{1}(s)\mbox{d}s]\subset \{|\xi|\leq r_{\infty}\}$  and
  there holds
  \begin{eqnarray}
    &&\left\|\int_{0}^{t}\mbox{e}^{-(t-s)A}F_{1}(s)\mbox{d}s\right\|_{H^{1}(0, T; L^{2})}
  +\left\|x\nabla\int_{0}^{t}\mbox{e}^{-(t-s)A}F_{1}(s)\mbox{d}s\right\|_{H^{1}(0, T; L^{2})}\nonumber\\
  &&\leq C\|F_{1}\|_{L^{2}(0, T; L^{2})}+\|x \nabla F_{1}\|_{L^{2}(0, T; L^{2})}\label{in4}
  \end{eqnarray}
  and  if in addition $F_{1}\in L^{2}(0, T; L_{1}^{2}),$  then
  \begin{eqnarray}
    \left\|\partial_{t}\int_{0}^{t}\mbox{e}^{-(t-s)A}F_{1}(s)\mbox{d}s\right\|_{L^{2}(0, T; L_{1}^{2})}
  \leq C\|F_{1}\|_{L^{2}(0, T; L_{1}^{2})}, \label{in5}
  \end{eqnarray}
  where in the above two inequalities, $C$  depends on $T.$
\end{enumerate}
\end{prop}

\begin{proof}
\begin{enumerate}
  \item
  The support properties can be easily verified
  and the three inequalities can be justified with the help of
  Plancherel theorem and the support property of $u_{1}.$
  We omit the details.
  \item
   Assertions in 1 and the support property of $F_{1}$
   leads to assertions in 2.
\end{enumerate}
We thus complete the proof of Proposition \ref{myprop2}.
\end{proof}

Next, we establish sufficient conditions that claim
the existence of $(1-\mbox{e}^{-TA})^{-1}.$
\begin{prop}\label{myprop3}
Let the conditions in Proposition \ref{myprop1} be satisfied.
Then   the following equation
\begin{equation}\label{inverse}
  (1-\mbox{e}^{-TA})u_{1}=F_{1}
\end{equation}
admits  a unique solution $u_{1}$ that satisfies $\operatorname{supp}\hat{u}_{1}\subset \{|\xi|\leq r_{\infty}\}$ and
\begin{equation*}
  \|u_{1}\|_{L^{2}}+\|x \nabla u_{1}\|_{L^{2}}\leq C\|F_{1}\|_{L_{1}^{1}}.
\end{equation*}
Consequently, $1-\mbox{e}^{-TA}$ has a bounded inverse
 $(1-\mbox{e}^{-TA})^{-1}: X_{1}\rightarrow X_{2}$,
 where $X_{1}=\{F_{1}: \operatorname{supp}\hat{F}_{1}\subset \{|\xi|\leq r_{\infty}\},
 F_{1}\in L^{2}\cap L_{1}^{1}
 \}$ and
 $X_{2}=\{u_{1}: \operatorname{supp}\hat{u}_{1}\subset \{|\xi|\leq r_{\infty}\},
 u_{1}\in L^{2}, x\nabla u_{1}\in L^{2}
 \}$ and there holds
  \begin{equation*}
  \|(1-\mbox{e}^{-tA})^{-1}F_{1}\|_{L^{2}}+\|x \nabla (1-\mbox{e}^{-tA})^{-1}F_{1}\|_{L^{2}}\leq C\|F_{1}\|_{L_{1}^{1}}.
\end{equation*}
\end{prop}
\begin{proof}
First, direct computation yields
\begin{equation*}
  |1-\mbox{e}^{-(1+i)\theta}|\geq  \sin\theta \geq c\theta
\end{equation*}
provided $\theta>0$ is sufficiently small.

This fact allows us to conclude that there exists a constant
$r_{\infty}>0$ such that
 \begin{equation}\label{r}
  |1-\mbox{e}^{-T\hat{A}}|=|1-\mbox{e}^{-T(1+i)|\xi|^{2}}|\leq   C \frac{1}{|\xi|^{2}}
\end{equation}
when $|\xi|\leq r_{\infty}.$

Applying Fourier transform to (\ref{inverse}), one obtains
\begin{equation}\label{inverse1}
  \hat{u}_{1}=(1-\mbox{e}^{-T\hat{A}})^{-1}\hat{F}_{1}.
\end{equation}
Thus, the support property of $u_{1}$ can be easily verified.

Plancherel theorem and (\ref{r}) yield
\begin{eqnarray*}
&&\|u_{1}\|_{L^{2}}
=\|\hat{u}_{1}\|_{L^{2}}
=\|(1-\mbox{e}^{-T\hat{A}})^{-1}\hat{F}_{1}\|_{L^{2}}
\leq C\left\|\frac{1}{|\xi|^{2}}\hat{F}_{1}\right\|_{L^{2}}
\leq C\left\|\frac{1}{|\xi|^{2}}(\hat{F}_{1}(\xi)-\hat{F}_{1}(0))\right\|_{L^{2}}\\
&&\leq C\left\|\frac{1}{|\xi|}\hat{F}_{1}\partial_{\xi}\hat{F}_{1}\right\|_{L^{2}}
\leq C\left\|\frac{1}{|\xi|}\right\|_{L^{2}(|\xi|\leq r_{\infty})}
\|F_{1}\|_{L^{1}}
\leq C\|F_{1}\|_{L^{1}},
\end{eqnarray*}
where $\hat{F}_{1}(0)=0$ since $F_{1}(-x)=-F_{1}(x).$

Note that $|\partial_{\xi} (1-\mbox{e}^{-T\hat{A}})^{-1}|\leq C\frac{1}{|\xi|^{3}}$.
This fact leads to
\begin{eqnarray*}
  \|x\nabla u_{1}\|_{L^{2}}
  =\|\partial_{\xi}(i\xi \hat{u}_{1})\|_{L^{2}}
  &\leq& C\| \hat{u}_{1}\|_{L^{2}}
  +C\| i\xi \partial_{\xi}\hat{u}_{1}\|_{L^{2}}\\
  &\leq& C\| \hat{u}_{1}\|_{L^{2}}
  +C\| i\xi \partial_{\xi} [(1-\mbox{e}^{-T\hat{A}})^{-1}\hat{F}_{1}]\|_{L^{2}}\\
  &\leq& C\| \hat{u}_{1}\|_{L^{2}}
  + C\| i\xi [\partial_{\xi} (1-\mbox{e}^{-T\hat{A}})^{-1}]\hat{F}_{1}\|_{L^{2}}\\
  &&+C\| i\xi  (1-\mbox{e}^{-T\hat{A}})^{-1}\partial_{\xi}\hat{F}_{1}\|_{L^{2}}\\
  &\leq& C\| \hat{u}_{1}\|_{L^{2}}
  +C\left\| \frac{1}{|\xi|^{2}}\hat{F}_{1}\right\|_{L^{2}}
  +C\left\| \frac{1}{|\xi|}\partial_{\xi}\hat{F}_{1}\right\|_{L^{2}}\\
  &\leq& C\| u_{1}\|_{L^{2}}
  +C\| F_{1}\|_{L^{1}}
  +C\|\partial_{\xi}\hat{F}_{1} \|_{L^{\infty}}\\
  &\leq& C\| u_{1}\|_{L^{2}}
  +C\| F_{1}\|_{L^{1}}
  +C\|x F_{1} \|_{L^{1}}.
\end{eqnarray*}
We thus complete the proof of Proposition \ref{myprop3}.
\end{proof}
We are now in a position to prove Proposition \ref{myprop1}.

\textbf{Proof of Proposition \ref{myprop1}:}
We only need to prove (\ref{lowin1}).

Using (\ref{in1}), (\ref{in2}), (\ref{in4}) and Proposition \ref{myprop3}, one finds
\begin{eqnarray*}
&&\left\|\mbox{e}^{-tA}(1-\mbox{e}^{-TA})^{-1}\int_{0}^{T}\mbox{e}^{-(T-s)A}F_{1}(s)\mbox{d}s\right\|_{H^{1}(0,T; L^{2})}\\
&&\quad+\left\|x \nabla\mbox{e}^{-tA}(1-\mbox{e}^{-TA})^{-1}\int_{0}^{T}\mbox{e}^{-(T-s)A}F_{1}(s)\mbox{d}s\right\|_{H^{1}(0,T; L^{2})}\\
&&\leq
C\left\|\int_{0}^{T}\mbox{e}^{-(T-s)A}(1-\mbox{e}^{-TA})^{-1}F_{1}(s)\mbox{d}s\right\|_{H^{1}(0,T; L^{2})}\\
&&\quad+C\left\|x \nabla\int_{0}^{T}\mbox{e}^{-(T-s)A}(1-\mbox{e}^{-TA})^{-1}F_{1}(s)\mbox{d}s\right\|_{H^{1}(0,T; L^{2})}\\
&&\leq C\|(1-\mbox{e}^{-TA})^{-1}F_{1}\|_{L^{2}(0,T; L^{2})}
+C\|x\nabla(1-\mbox{e}^{-TA})^{-1}F_{1}\|_{L^{2}(0,T; L^{2})}\\
&&\leq C\|F_{1}\|_{L^{2}(0,T; L_{1}^{1})}.
\end{eqnarray*}
Employing (\ref{in3}) and (\ref{in4}) yields
\begin{eqnarray*}
&&\left\|\partial_{t}\mbox{e}^{-tA}(1-\mbox{e}^{-TA})^{-1}\int_{0}^{T}\mbox{e}^{-(T-s)A}F_{1}(s)\mbox{d}s\right\|_{L^{2}(0,T; L_{1}^{2})}\\
&&\leq
C\left\|\int_{0}^{T}\mbox{e}^{-(T-s)A}(1-\mbox{e}^{-TA})^{-1}F_{1}(s)\mbox{d}s\right\|_{L^{2}(0,T; L^{2})}\\
&&\quad+C\left\|x \nabla\int_{0}^{T}\mbox{e}^{-(T-s)A}(1-\mbox{e}^{-TA})^{-1}F_{1}(s)\mbox{d}s\right\|_{L^{2}(0,T; L^{2})}\\
&&\leq C\|(1-\mbox{e}^{-TA})^{-1}F_{1}\|_{L^{2}(0,T; L^{2})}
+C\|x\nabla(1-\mbox{e}^{-TA})^{-1}F_{1}\|_{L^{2}(0,T; L^{2})}\\
&&\leq C\|F_{1}\|_{L^{2}(0,T; L_{1}^{1})}.
\end{eqnarray*}

Invoking (\ref{in4}), (\ref{in5})  and Lemma \ref{lem403}(ii), we obtain
\begin{eqnarray*}
    &&\left\|\int_{0}^{t}\mbox{e}^{-(t-s)A}F_{1}(s)\mbox{d}s\right\|_{H^{1}(0, T; L^{2})}
  +\left\|x\nabla\int_{0}^{t}\mbox{e}^{-(t-s)A}F_{1}(s)\mbox{d}s\right\|_{H^{1}(0, T; L^{2})}\\
  &&\quad+\left\|\partial_{t}\int_{0}^{t}\mbox{e}^{-(t-s)A}F_{1}(s)\mbox{d}s\right\|_{L^{2}(0, T; L_{1}^{2})}
  \\
  &&\leq C\|F_{1}\|_{L^{2}(0, T; L^{2})}+\|x \nabla F_{1}\|_{L^{2}(0, T; L^{2})}
  + C\|F_{1}\|_{L^{2}(0, T; L_{1}^{2})}\\
  &&\leq C\|F_{1}\|_{L^{2}(0,T; L_{1}^{1})}.
  \end{eqnarray*}
We thus complete the proof of Proposition \ref{myprop1}.

\section{Estimates of the high frequency part }

The following high frequency  equation are investigated in this section
\begin{eqnarray}\label{glhigh}
\begin{cases}
\partial_{t}u_{\infty}+A u_{\infty}=F_{\infty},\\
u_{\infty}|_{t=0}=u_{\infty0},\\
u_{\infty}(0)=u_{\infty}(T),
\end{cases}
\end{eqnarray}
where $F_{\infty}$ is assumed to satisfy $\operatorname{supp}\hat{F}_{\infty}\subset\{ |\xi|\geq r_{1}\}$.

Similar to the derivation of (\ref{formulalow1}),
 we formally have
\begin{eqnarray}\label{formulahigh1}
u_{\infty}(t)=\mbox{e}^{-tA}(1-\mbox{e}^{-TA})^{-1}\int_{0}^{T}\mbox{e}^{-(T-s)A}F_{\infty}(s)\mbox{d}s
+\int_{0}^{t}\mbox{e}^{-(t-s)A}F_{\infty}(s)\mbox{d}s.
\end{eqnarray}

The following Proposition is concerned with the solvability of (\ref{glhigh}).

\begin{prop}\label{prop6050}
Assume that $F_{\infty}\in L^{2}(0, T; X_{\infty}^{1})$
and $F_{\infty}(-x)=-F_{\infty}(x)$,
where
  $X_{\infty}^{k}=\{f: \operatorname{supp}\hat{f}\subset\{|\xi|\geq r_{1}\}, f\in H_{1}^{k}\}$.
Then $u_{\infty}(t)$ defined by (\ref{formulahigh1}) is a solution of (\ref{glhigh})
in $Y(0, T)$
and there holds
\begin{equation*}
  \|u_{\infty}\|_{Y(0, T)}\leq C\|F_{\infty}\|_{L^{2}(0, T; H_{1}^{1})}.
\end{equation*}
\end{prop}

To prove Proposition \ref{prop6050},
we first establish the following weighted energy estimates.

\begin{prop}\label{prop606}
For the smooth solution of equation (\ref{glhigh}),
there holds the following estimate
\begin{eqnarray}\label{inhigh}
\frac{1}{2}\frac{\mbox{d}}{\mbox{d}t}\| u_{\infty}\|_{H^{2}_{1}}^{2}
+d\| u_{\infty}\|_{H^{3}_{1}}^{2}
\leq
C
\|F_{\infty}\|_{H^{1}_{1}}^{2},
\end{eqnarray}
where $d$ is a positive constant  large enough.
\end{prop}
\begin{proof}
The standard  energy estimates yields
\begin{eqnarray}\label{energy1}
\frac{1}{2}\frac{\mbox{d}}{\mbox{d}t}\|u_{\infty}\|_{L^{2}}^{2}
+\|\nabla u_{\infty}\|_{L^{2}}^{2}
\leq \frac{1}{\epsilon}\|F_{\infty}\|_{L^{2}}^{2}+\epsilon\|u_{\infty}\|_{L^{2}}^{2}
\end{eqnarray}
and
\begin{eqnarray}\label{energy2}
\frac{1}{2}\frac{\mbox{d}}{\mbox{d}t}\|\nabla u_{\infty}\|_{H^{1}}^{2}
+\|\nabla^{2} u_{\infty}\|_{H^{1}}^{2}
\leq
C\|F_{\infty}\|_{H^{1}}^{2}.
\end{eqnarray}
where we have used integration by parts in the derivation of  (\ref{energy2}).

Combining (\ref{energy1}), (\ref{energy2}) and Lemma \ref{lem404}(ii),
one deduces after choosing $\epsilon$ small enough that
\begin{eqnarray}\label{energy3}
\frac{1}{2}\frac{\mbox{d}}{\mbox{d}t}\| u_{\infty}\|_{H^{2}}^{2}
+\|\nabla u_{\infty}\|_{H^{2}}^{2}
\leq
C\|F_{\infty}\|_{H^{1}}^{2}.
\end{eqnarray}
Applying $\partial_{x}^{\alpha}$ to $(\ref{glhigh})_{1}$ for $0\leq|\alpha|\leq 2$
and dotting the result with $\partial_{x}^{\alpha}\bar{u}_{\infty}$, one obtains
\begin{eqnarray}\label{wenergy1}
&&\frac{1}{2}\frac{\mbox{d}}{\mbox{d}t}\| |x|\partial_{x}^{\alpha}u_{\infty}\|_{L^{2}}^{2}
+\| |x|\partial_{x}^{\alpha}\nabla u_{\infty}\|_{L^{2}}^{2}\nonumber\\
&&=-Re\left\{
(1+i)\int_{\mathbb{R}^{3}}\nabla(|x|^{2})\nabla\partial_{x}^{\alpha} u_{\infty}
\partial_{x}^{\alpha} \bar{u}_{\infty}\mbox{d}x
\right\}
+Re\int_{\mathbb{R}^{3}}|x|^{2}\partial_{x}^{\alpha} F_{\infty}\partial_{x}^{\alpha} \bar{u}_{\infty}\mbox{d}x\nonumber\\
&&=A+B.
\end{eqnarray}
When $|\alpha|=0,$ we find
\begin{eqnarray*}%\label{wenergy1a1}
A\leq
\epsilon_{1} |u_{\infty}|_{L^{2}_{1}}^{2}+C\frac{1}{\epsilon_{1}}\|\nabla u_{\infty}\|_{L^{2}}^{2}
\end{eqnarray*}
and
\begin{eqnarray*}%\label{wenergy1b1}
B
\leq
\epsilon_{1}|u_{\infty}|_{L^{2}_{1}}^{2}+C\frac{1}{\epsilon_{1}}|F_{\infty}|_{L^{2}_{1}},
\end{eqnarray*}
where $|f|_{L^{p}_{1}}=\|xf\|_{L^{p}}.$

When $|\alpha|\geq1,$ one derives
\begin{eqnarray*}%\label{wenergy1a2}
A
\leq
\|x\nabla \partial_{x}^{\alpha}u_{\infty}\|_{L^{2}}\| \partial_{x}^{\alpha}u_{\infty}\|_{L^{2}}
\leq
\epsilon_{2} |\nabla^{2}u_{\infty}|_{H^{1}_{1}}^{2}
+C\frac{1}{\epsilon_{2}} \|\nabla u_{\infty}\|_{H^{1}}^{2}
\end{eqnarray*}
and
\begin{eqnarray*}%\label{wenergy1b2}
&&B
=
-Re\int_{\mathbb{R}^{3}}\nabla(|x|^{2})\partial_{x}^{\alpha-1}F_{\infty}
\partial_{x}^{\alpha}\bar{u}_{\infty}\mbox{d}x
-Re\int_{\mathbb{R}^{3}}|x|^{2}\partial_{x}^{\alpha-1}F_{\infty}\partial_{x}^{\alpha+1}\bar{u}_{\infty}\mbox{d}x\nonumber\\
&&\quad\leq
\|x \partial_{x}^{\alpha-1}F_{\infty}\|_{L^{2}}\| \partial_{x}^{\alpha}u_{\infty}\|_{L^{2}}
+\|x \partial_{x}^{\alpha-1}F_{\infty}\|_{L^{2}}\| x\partial_{x}^{\alpha+1}u_{\infty}\|_{L^{2}}\nonumber\\
&&\quad\leq
\frac{1}{2}|F_{\infty}|_{H^{1}_{1}}^{2}
+\frac{1}{2}\|u_{\infty}\|_{H^{2}}^{2}
+\epsilon_{2}|\nabla u_{\infty}|_{H^{2}_{1}}^{2}
+C\frac{1}{\epsilon_{2}}|F_{\infty}|_{H^{1}_{1}}^{2},
\end{eqnarray*}
where $|f|_{H^{k}_{1}}=\|xf\|_{H^{k}}.$

Substituting the above estimates into (\ref{wenergy1}), we obtain
\begin{eqnarray*}%\label{wenergy2}
&&\frac{1}{2}\frac{\mbox{d}}{\mbox{d}t}| u_{\infty}|_{H^{2}_{1}}^{2}
+|\nabla u_{\infty}|_{H^{2}_{1}}^{2}\nonumber\\
&&\leq
\epsilon_{1} |u_{\infty}|_{L^{2}_{1}}^{2}
+\epsilon_{2} |\nabla u_{\infty}|_{H^{2}_{1}}^{2}
+C\frac{1}{\epsilon_{1}}\|u_{\infty}\|_{H^{2}}^{2}
+C\|u_{\infty}\|_{H^{2}}^{2}\nonumber\\
&&\quad+C\frac{1}{\epsilon_{2}}\|u_{\infty}\|_{H^{2}}^{2}
+C(1+\frac{1}{\epsilon_{1}})|F_{\infty}|_{H^{1}_{1}}^{2}
+C\frac{1}{\epsilon_{2}}|F_{\infty}|_{H^{1}_{1}}^{2}.
\end{eqnarray*}
Choosing $\epsilon_{2}$ small enough and employing Lemmas \ref{lem404}(ii) and \ref{lem406} yield
\begin{eqnarray}\label{wenergy3}
&&\frac{1}{2}\frac{\mbox{d}}{\mbox{d}t}| u_{\infty}|_{H^{2}_{1}}^{2}
+|\nabla u_{\infty}|_{H^{2}_{1}}^{2}\nonumber\\
&&\leq
C[\epsilon_{1} |u_{\infty}|_{L^{2}_{1}}^{2}
+(\frac{1}{\epsilon_{1}}+1)\|u_{\infty}\|_{H^{2}}^{2}
+(\frac{1}{\epsilon_{1}}+1)|F_{\infty}|_{H^{1}_{1}}^{2}]\nonumber\\
&&\leq C[\epsilon_{1} \|x \nabla u_{\infty}\|_{L^{2}}^{2}
+\epsilon_{1}\| \nabla u_{\infty}\|_{L^{2}}^{2}
+(\frac{1}{\epsilon_{1}}+1)\|u_{\infty}\|_{H^{2}}^{2}
+(\frac{1}{\epsilon_{1}}+1)|F_{\infty}|_{H^{1}_{1}}^{2}].
\end{eqnarray}

Choosing $\epsilon_{1}$ small enough, then adding
(\ref{wenergy3}) to $d_{1}\times (\ref{energy3})$
with $d_{1}$ sufficiently large,
one deduces
\begin{eqnarray*}
\frac{1}{2}\frac{\mbox{d}}{\mbox{d}t}(d_{1}\| u_{\infty}\|_{H^{2}_{1}}^{2})
+d\| u_{\infty}\|_{H^{3}_{1}}^{2}
\leq
C
\|F_{\infty}\|_{H^{1}_{1}}^{2}.
\end{eqnarray*}
We thus complete the proof of Proposition \ref{prop606}.
\end{proof}
Based on the weighted energy estimates provided by Proposition \ref{prop606},
we next show the existence of $(1-\mbox{e}^{-TA})^{-1}$ in high frequency.
\begin{prop}\label{prop605}
We claim that:
\begin{enumerate}
  \item
  If $\operatorname{supp}\hat{u}_{\infty0}\subset\{|\xi|\geq r_{1}\}$ and
  $u_{\infty0}\in H_{1}^{2}$, then
  $\operatorname{supp}\mathcal{F}[\mbox{e}^{-tA}u_{\infty0}]\subset\{|\xi|\geq r_{1}\}$ and
  \begin{equation*}
    \|\mbox{e}^{-tA}u_{\infty0}\|_{H_{1}^{k}}\leq C\mbox{e}^{-at}\|u_{\infty0}\|_{H_{1}^{2}}
  \end{equation*}
  for all $t\geq0.$
  \item If $\operatorname{supp}\hat{F}_{\infty}\subset\{|\xi|\geq r_{1}\}$ and
  $F_{\infty}\in L^{2}(0, T;H_{1}^{1})$, then
  $\operatorname{supp}\mathcal{F}[\mbox{e}^{-tA}F_{\infty}]\subset\{|\xi|\geq r_{1}\}$ and
  \begin{equation*}
    \left\|\int_{0}^{t}\mbox{e}^{-(t-s)A}F_{\infty}(s)\mbox{d}s\right\|_{H_{1}^{2}}\leq C\left\{\int_{0}^{t}e^{-a(t-s)}\|F_{\infty}\|^{2}_{H_{1}^{1}}\mbox{d}s\right\}^{\frac{1}{2}}
  \end{equation*}
  for $t\in [0, T]$ with $C$ depending on $T.$
  \item
  The spectral radius of $\mbox{e}^{-TA}$ in $X_{\infty}^{2}$ is less than 1.
  Accordingly, $1-\mbox{e}^{-TA}$ has a bounded inverse $(1-\mbox{e}^{-TA})^{-1}$ on
  $X_{\infty}$ and there holds
  \begin{equation*}
    \|(1-\mbox{e}^{-TA})^{-1}u\|_{H_{1}^{2}}\leq C\|u\|_{H_{1}^{2}}
  \end{equation*}
  provided that $u\in X_{\infty}^{2}.$
\end{enumerate}
\end{prop}

\begin{proof}
\begin{enumerate}
  \item
  From (\ref{inhigh}) we derive
  \begin{eqnarray}\label{inhigh1}
\frac{1}{2}\frac{\mbox{d}}{\mbox{d}t}\| u_{\infty}\|_{H^{2}_{1}}^{2}
+d\| u_{\infty}\|_{H^{2}_{1}}^{2}
\leq
C
\|F_{\infty}\|_{H^{1}_{1}}^{2},
\end{eqnarray}
Let $F_{\infty}=0$ in (\ref{inhigh1}), then one finds
\begin{eqnarray*}
\| u_{\infty}\|_{H^{2}_{1}}^{2}
\leq
\mbox{e}^{-2dt}
\| u_{\infty0}\|_{H^{2}_{1}}^{2},
\end{eqnarray*}
So there exists $a=d$ such that
\begin{eqnarray*}
\| u_{\infty}\|_{H^{2}_{1}}
\leq
\mbox{e}^{-at}
\| u_{\infty0}\|_{H^{2}_{1}}.
\end{eqnarray*}
This proves claim 1.
  \item
  The inequality in claim 2 is proved by setting $ u_{\infty0}=0$ in (\ref{inhigh1}).
  \item
  Let $\operatorname{supp}\hat{u}\subset \{|\xi|\geq r_{1}\}$.
  From claim 1, we know
  \begin{equation*}
    \|(\mbox{e}^{-TA})^{j}u\|_{H_{1}^{2}}
    =\|\mbox{e}^{-jTA}u\|_{H_{1}^{2}}
    \leq \mbox{e}^{-jat}
\| u\|_{H^{2}_{1}}.
  \end{equation*}
Consequently, we find $ \|(\mbox{e}^{-TA})^{j}\|\leq \mbox{e}^{-jat}$.
This leads to
\begin{equation*}
  \lim_{j\rightarrow \infty}\|(\mbox{e}^{-TA})^{j}\|^{\frac{1}{j}}
  \leq \lim_{j\rightarrow \infty}\mbox{e}^{-at} <1.
\end{equation*}
\end{enumerate}
We thus complete the proof of Proposition \ref{prop605}.
\end{proof}

We are in a position to prove Proposition \ref{prop6050}.

\textbf{Proof of Proposition \ref{prop6050}:}

Invoking (\ref{inhigh}) yields
\begin{eqnarray}\label{6071}
&&\|u_{\infty}\|_{H_{1}^{2}}^{2}
+\|u_{\infty}\|_{L^{2}(0,T;H_{1}^{3})}^{2}\nonumber\\
&&\leq C\{
\|u_{0\infty}\|_{H_{1}^{2}}^{2}
+\|F_{\infty}\|_{L^{2}(0,T;H_{1}^{1})}^{2}
\}\nonumber\\
&&\leq C\left\{
\left\|(1-\mbox{e}^{-TA})^{-1}\int_{0}^{T}
\mbox{e}^{-(T-s)A}F_{\infty}\mbox{d}s\right\|_{H_{1}^{2}}^{2}
+\|F_{\infty}\|_{L^{2}(0,T;H_{1}^{1})}^{2}
\right\}.
\end{eqnarray}
According to claim 3 and claim 2 of Proposition \ref{prop605}, one finds
\begin{eqnarray}\label{6072}
(\ref{6071})
&\leq& C\left\{
\left\|\int_{0}^{T}
\mbox{e}^{-(T-s)A}F_{\infty}\mbox{d}s\right\|_{H_{1}^{2}}^{2}
+\|F_{\infty}\|_{L^{2}(0,T;H_{1}^{1})}^{2}
\right\}\nonumber\\
&\leq& C
\|F_{\infty}\|_{L^{2}(0,T;H_{1}^{1})}^{2}.
\end{eqnarray}
Note that $u_{\infty}$ satisfies (\ref{glhigh}).
This leads to
\begin{eqnarray*}
\|\partial_{t}u_{\infty}\|_{L^{2}(0, T;H_{1}^{1})}
&\leq& C\|u_{\infty}\|_{L^{2}(0, T;H_{1}^{3})}
+C\|F_{\infty}\|_{L^{2}(0, T;H_{1}^{1})}\nonumber\\
&\leq& C\|F_{\infty}\|_{L^{2}(0, T;H_{1}^{1})}.\quad \text{(by (\ref{6071}) and (\ref{6072}))}
\end{eqnarray*}
We thus complete the proof of Proposition \ref{prop6050}.

\section{Estimates of the nonlinear and inhomogeneous terms}
This section deals with the nonlinear and inhomogeneous terms.
We have the following Proposition.
\begin{prop}\label{prop801}
For $u=u_{1}+u_{\infty},$ there holds the following estimates
\begin{eqnarray}
&&\|F_{1}(u,g)\|_{L^{2}(0, T;L^{1}_{1})}
\leq C\|\{u_{1},u_{\infty}\}\|_{Z(0, T)}^{3}+C\|g\|_{L^{2}(0, T;L^{1}_{1})},\label{f1ug}\\
&&\|F_{\infty}(u,g)\|_{L^{2}(0, T;H^{1}_{1})}
\leq C\|\{u_{1},u_{\infty}\}\|_{Z(0, T)}^{3}+C\|g\|_{L^{2}(0, T;H^{1}_{1})},\label{f2ug}\\
&&\|F_{1}(u,g)-F_{1}(v,g)\|_{L^{2}(0, T;L^{1}_{1})}
\leq C(\|\{u_{1},u_{\infty}\}\|_{Z(0, T)}^{2}
+\|\{v_{1},v_{\infty}\}\|_{Z(0, T)}^{2})\nonumber\\
&&\qquad\qquad\qquad\qquad\qquad\qquad\qquad\qquad\times(\|\{u_{1}-v_{1},u_{\infty}-v_{\infty}\}\|_{Z(0, T)}),\label{f3ug}\\
&&\|F_{\infty}(u,g)-F_{\infty}(v,g)\|_{L^{2}(0, T;H^{1}_{1})}
\leq C(\|\{u_{1},u_{\infty}\}\|_{Z(0, T)}^{2}
+\|\{v_{1},v_{\infty}\}\|_{Z(0, T)}^{2})\nonumber\\
&&\qquad\qquad\qquad\qquad\qquad\qquad\qquad\qquad\times(\|\{u_{1}-v_{1},u_{\infty}-v_{\infty}\}\|_{Z(0, T)}).\label{f4ug}
\end{eqnarray}
\end{prop}

\begin{proof}
First, one easily derives from Lemma \ref{lem403} that
\begin{eqnarray}\label{nonlinear1}
&&\|P_{1}(|u|^{2}u)\|_{L^{1}_{1}}
\leq
C(\|u_{1}^{3}\|_{L^{1}_{1}}
+\|u_{1}^{2}u_{\infty}\|_{L^{1}_{1}}
+\|u_{\infty}^{2}u_{1}\|_{L^{1}_{1}}
+\|u_{\infty}^{3}\|_{L^{1}_{1}}
)\nonumber\\
&&\qquad\qquad\qquad=I_{1}+I_{2}+I_{3}+I_{4}.
\end{eqnarray}

Using interpolation and $\|f\|_{L^{6}(\mathbb{R}^{3})}\leq C\|\nabla f\|_{L^{2}(\mathbb{R}^{3})}$,
we obtain
\begin{eqnarray*}
&&I_{1}
=C\|u_{1}^{3}\|_{L^{1}_{1}}\nonumber\\
&&\leq
C\|u_{1}^{3}\|_{L^{1}}+C\|xu_{1}^{3}\|_{L^{1}}\nonumber\\
&&\leq
C\|u_{1}\|_{L^{2}}\|u_{1}\|_{L^{3}}\|u_{1}\|_{L^{6}}
+C\|u_{1}\|_{L^{2}}\|u_{1}\|_{L^{3}}\|xu_{1}\|_{L^{6}}\nonumber\\
&&\leq
C\|u_{1}\|_{L^{2}}\|u_{1}\|_{L^{2}}^{\frac{1}{2}}\|u_{1}\|_{L^{6}}^{\frac{1}{2}}\|u_{1}\|_{L^{6}}
+C\|u_{1}\|_{L^{2}}\|u_{1}\|_{L^{2}}^{\frac{1}{2}}\|u_{1}\|_{L^{6}}^{\frac{1}{2}}\|u_{1}\|_{L^{6}}\|xu_{1}\|_{L^{6}}\nonumber\\
&&\leq
C\|u_{1}\|_{L^{2}}^{\frac{3}{2}}\|\nabla u_{1}\|_{L^{2}}^{\frac{3}{2}}
+C\|u_{1}\|_{L^{2}}^{\frac{3}{2}}\|\nabla u_{1}\|_{L^{2}}^{\frac{1}{2}}(\|u_{1}\|_{L^{2}}+\| x\nabla u_{1}\|_{L^{2}})\nonumber\\
&&\leq
C\|u_{1}\|_{L^{2}}^{3}+C\|u_{1}\|_{L^{2}}^{2}\|x\nabla u_{1}\|_{L^{2}},
\end{eqnarray*}
where in the last inequality, we have invoked Lemma \ref{lem403}.

Similarly, it finds
\begin{eqnarray*}
&&I_{2}
=C\|u_{1}^{2}u_{\infty}\|_{L^{1}_{1}}\nonumber\\
&&\leq
C\|u_{1}^{2}u_{\infty}\|_{L^{1}}+C\|xu_{1}^{2}u_{\infty}\|_{L^{1}}\nonumber\\
&&\leq
C\|u_{\infty}\|_{L^{2}}\|u_{1}\|_{L^{3}}\|u_{1}\|_{L^{6}}
+C\|u_{\infty}\|_{L^{2}}\|u_{1}\|_{L^{3}}\|xu_{1}\|_{L^{6}}\nonumber\\
&&\leq
C\|u_{\infty}\|_{L^{2}}\|u_{1}\|_{L^{2}}^{\frac{1}{2}}\|u_{1}\|_{L^{6}}^{\frac{3}{2}}
+C\|u_{\infty}\|_{L^{2}}\|u_{1}\|_{L^{2}}^{\frac{1}{2}}\|u_{1}\|_{L^{6}}^{\frac{1}{2}}\|\nabla (xu_{1})\|_{L^{2}}\nonumber\\
&&\leq
C\|u_{\infty}\|_{L^{2}}\| u_{1}\|_{L^{2}}^{\frac{1}{2}}\|\nabla u_{1}\|_{L^{2}}^{\frac{3}{2}}
+C\|u_{\infty}\|_{L^{2}}\| u_{1}\|_{L^{2}}^{\frac{1}{2}}\|\nabla u_{1}\|_{L^{2}}^{\frac{1}{2}}(\|u_{1}\|_{L^{2}}+\| x\nabla u_{1}\|_{L^{2}})\nonumber\\
&&\leq
C\|u_{\infty}\|_{L^{2}}\| u_{1}\|_{L^{2}}^{2}+C\|u_{\infty}\|_{L^{2}}\|x\nabla u_{1}\|_{L^{2}},
\end{eqnarray*}
\begin{eqnarray*}
&&I_{3}
=C\|u_{\infty}^{2}u_{1}\|_{L^{1}_{1}}\nonumber\\
&&\leq
C\|u_{\infty}^{2}u_{1}\|_{L^{1}}+C\|xu_{\infty}^{2}u_{1}\|_{L^{1}}\nonumber\\
&&\leq
C\|u_{\infty}\|_{L^{2}}\|u_{\infty}\|_{L^{3}}(\|u_{1}\|_{L^{6}}
+\|xu_{1}\|_{L^{6}})\nonumber\\
&&\leq
C\|u_{\infty}\|_{L^{2}}\|u_{\infty}\|_{L^{2}}^{\frac{1}{2}}\|\nabla u_{\infty}\|_{L^{2}}^{\frac{1}{2}}
(\|\nabla u_{1}\|_{L^{2}}
+\|\nabla (xu_{1})\|_{L^{2}})\nonumber\\
&&\leq
C\|u_{\infty}\|_{L^{2}}^{\frac{3}{2}}\|\nabla u_{\infty}\|_{L^{2}}^{\frac{1}{2}}
(\| u_{1}\|_{L^{2}}
+\|x\nabla u_{1}\|_{L^{2}}),
\end{eqnarray*}
and
\begin{eqnarray*}
&&I_{4}
=C\|u_{\infty}^{3}\|_{L^{1}_{1}}\nonumber\\
&&\leq
C\|u_{\infty}^{3}\|_{L^{1}}+C\|xu_{\infty}^{3}\|_{L^{1}}\nonumber\\
&&\leq
C\|u_{\infty}\|_{L^{2}}\|u_{\infty}\|_{L^{3}}(\|u_{\infty}\|_{L^{6}}
+\|xu_{\infty}\|_{L^{6}})\nonumber\\
&&\leq
C\|u_{\infty}\|_{L^{2}}^{\frac{3}{2}}\|\nabla u_{\infty}\|_{L^{2}}^{\frac{1}{2}}
(\| \nabla u_{\infty}\|_{L^{2}}
+\|x \nabla u_{\infty}\|_{L^{2}}).
\end{eqnarray*}
Inserting the above estimates of $I_{1}-I_{4}$ to
(\ref{nonlinear1}), one obtains (\ref{f1ug}).

Employing Lemmas \ref{lem404} and \ref{lem405}, we easily find
\begin{eqnarray}\label{nonlinear2}
&&\|P_{\infty}(|u|^{2}u)\|_{H^{1}_{1}}
\leq
C(\|u_{1}^{3}\|_{H^{1}_{1}}
+\|u_{1}^{2}u_{\infty}\|_{H^{1}_{1}}
+\|u_{\infty}^{2}u_{1}\|_{H^{1}_{1}}
+\|u_{\infty}^{3}\|_{H^{1}_{1}}
)\nonumber\\
&&\qquad\qquad\qquad\qquad=J_{1}+J_{2}+J_{3}+J_{4}.
\end{eqnarray}
Lemma (\ref{lem403}) again leads to
\begin{eqnarray*}
&&J_{1}
=C\|u_{1}^{3}\|_{H^{1}_{1}}\nonumber\\
&&\leq
C\|u_{1}^{3}\|_{H^{1}}+C\|xu_{1}^{3}\|_{H^{1}}\nonumber\\
&&\leq C\|u_{1}^{3}\|_{L^{2}}
+C\|\nabla (u_{1}^{3})\|_{L^{2}}
+C\|x u_{1}^{3}\|_{L^{2}}
+C\|\nabla (x u_{1}^{3})\|_{L^{2}}\nonumber\\
&&\leq C\|u_{1}\|_{L^{6}}^{3}
+C\|u_{1}^{2}\nabla u_{1}\|_{L^{2}}
+C\|x u_{1}\|_{L^{6}}\| u_{1}\|_{L^{\infty}}\| u_{1}\|_{L^{3}}
+C\|x u_{1}^{2}\nabla u_{1}\|_{L^{2}}\nonumber\\
&&\leq C\|u_{1}\|_{L^{2}}^{3}
+C\|u_{1}\|_{L^{\infty}}^{2}\|\nabla u_{1}\|_{L^{2}}
+(\|u_{1}\|_{L^{2}}+\|x \nabla u_{1}\|_{L^{2}})\|u_{1}\|_{L^{2}}\|u_{1}\|_{L^{2}}^{\frac{1}{2}}\|\nabla u_{1}\|_{L^{2}}^{\frac{1}{2}}\nonumber\\
&&\quad+\|u_{1}\|_{L^{\infty}}^{2}\|x \nabla u_{1}\|_{L^{2}}\nonumber\\
&&\leq C\|u_{1}\|_{L^{2}}^{3}+C\|u_{1}\|_{L^{2}}^{2}\|x \nabla u_{1}\|_{L^{2}},
\end{eqnarray*}

\begin{eqnarray*}
&&J_{2}
=C\|u_{1}^{2}u_{\infty}\|_{H^{1}_{1}}\nonumber\\
&&\leq
C\|u_{1}^{2}u_{\infty}\|_{H^{1}}+C\|xu_{1}^{2}u_{\infty}\|_{H^{1}}\nonumber\\
&&\leq C\|u_{1}^{2}u_{\infty}\|_{L^{2}}
+C\|\nabla (u_{1}^{2}u_{\infty})\|_{L^{2}}
+C\|x u_{1}^{2}u_{\infty}\|_{L^{2}}
+C\|\nabla (x u_{1}^{2}u_{\infty})\|_{L^{2}}\nonumber\\
&&\leq C\|u_{1}\|_{L^{\infty}}^{2}\|u_{\infty}\|_{L^{2}}
+C\|u_{1}^{2}\nabla u_{\infty}\|_{L^{2}}
+C\|u_{1}u_{\infty}\nabla u_{1}\|_{L^{2}}
+C\|x u_{\infty}\|_{L^{2}}\| u_{1}\|_{L^{\infty}}^{2}\nonumber\\
&&\quad+C\|x u_{1}^{2}\nabla u_{\infty}\|_{L^{2}}
+C\|x u_{1}u_{\infty}\nabla u_{1}\|_{L^{2}}\nonumber\\
&&\leq C\|u_{1}\|_{L^{2}}^{2}\|u_{\infty}\|_{L^{2}}
+C\|u_{1}\|_{L^{\infty}}^{2}\|\nabla u_{\infty}\|_{L^{2}}
+C\|u_{1}\|_{L^{\infty}}\|u_{\infty}\|_{L^{\infty}}\|\nabla u_{1}\|_{L^{2}}\nonumber\\
&&\quad+C\|x u_{\infty}\|_{L^{2}}\| u_{1}\|_{L^{\infty}}^{2}
+C\|x\nabla  u_{\infty}\|_{L^{2}}\| u_{1}\|_{L^{\infty}}^{2}
+\|u_{1}\|_{L^{\infty}}\|u_{\infty}\|_{L^{\infty}}\|x \nabla u_{1}\|_{L^{2}}
\nonumber\\
&&\leq C\|u_{1}\|_{L^{2}}^{2}(\|u_{\infty}\|_{L^{2}}+\| u_{\infty}\|_{H^{2}})
+C\|u_{1}\|_{L^{2}}^{2}(\|x  u_{\infty}\|_{L^{2}}+\|x \nabla u_{\infty}\|_{L^{2}})\nonumber\\
&&\quad+C\|u_{1}\|_{L^{2}}\|u_{\infty}\|_{H^{2}}\|x \nabla u_{1}\|_{L^{2}},
\end{eqnarray*}

\begin{eqnarray*}
&&J_{3}
=C\|u_{\infty}^{2}u_{1}\|_{H^{1}_{1}}\nonumber\\
&&\leq
C\|u_{\infty}^{2}u_{1}\|_{H^{1}}+C\|xu_{\infty}^{2}u_{1}\|_{H^{1}}\nonumber\\
&&\leq C\|u_{\infty}^{2}u_{1}\|_{L^{2}}
+C\|\nabla (u_{1}u_{\infty}^{2})\|_{L^{2}}
+C\|x u_{1}u_{\infty}^{2}\|_{L^{2}}
+C\|\nabla (x u_{1}u_{\infty}^{2})\|_{L^{2}}\nonumber\\
&&\leq C\|u_{\infty}\|_{L^{\infty}}^{2}\|u_{1}\|_{L^{2}}
+C\|u_{\infty}^{2}\nabla u_{1}\|_{L^{2}}
+C\|u_{1}u_{\infty}\nabla u_{\infty}\|_{L^{2}}\nonumber\\
&&\quad+C\|x u_{\infty}\|_{L^{2}}\| u_{1}\|_{L^{\infty}}\| u_{\infty}\|_{L^{\infty}}
+C\|x u_{\infty}^{2}\nabla u_{1}\|_{L^{2}}
+C\|x u_{1}u_{\infty}\nabla u_{\infty}\|_{L^{2}}\nonumber\\
&&\leq C\|u_{1}\|_{L^{2}}\|u_{\infty}\|_{H^{2}}^{2}
+C\|u_{1}\|_{L^{\infty}}\|u_{\infty}\|_{L^{\infty}}\|\nabla u_{\infty}\|_{L^{2}}
+C\|u_{1}\|_{L^{\infty}}\|u_{\infty}\|_{L^{\infty}}\|x  u_{\infty}\|_{L^{2}}\nonumber\\
&&\quad+\|u_{\infty}\|_{L^{\infty}}\|x \nabla u_{1}\|_{L^{2}}
+\|u_{1}\|_{L^{\infty}}\|u_{\infty}\|_{L^{\infty}}\|x \nabla u_{\infty}\|_{L^{2}}
\nonumber\\
&&\leq C\|u_{1}\|_{L^{2}}\|u_{\infty}\|_{H^{2}}^{2}
+C\|u_{1}\|_{L^{2}}\|u_{\infty}\|_{H^{2}}\|x  u_{\infty}\|_{L^{2}}\nonumber\\
&&\quad+\|u_{\infty}\|_{H^{2}}\|x \nabla u_{1}\|_{L^{2}}
+\|u_{1}\|_{L^{2}}\|u_{\infty}\|_{H^{2}}\|x \nabla u_{\infty}\|_{L^{2}},
\end{eqnarray*}
and
\begin{eqnarray*}
&&J_{4}
=C\|u_{\infty}^{3}\|_{H^{1}_{1}}\nonumber\\
&&\leq
C\|u_{\infty}^{3}\|_{H^{1}}+C\|xu_{\infty}^{3}\|_{H^{1}}\nonumber\\
&&\leq C\|u_{\infty}^{3}\|_{L^{2}}
+C\|\nabla (u_{\infty}^{3})\|_{L^{2}}
+C\|x u_{\infty}^{3}\|_{L^{2}}
+C\|\nabla (x u_{\infty}^{3})\|_{L^{2}}\nonumber\\
&&\leq C\|u_{\infty}\|_{H^{2}}^{3}
+C\|u_{\infty}\|_{H^{2}}^{2}\|x u_{\infty}\|_{L^{2}}
+C \|u_{\infty}\|_{H^{2}}^{2}\|x \nabla u_{\infty}\|_{L^{2}}.
\end{eqnarray*}
Substituting the above estimates of $J_{1}-J_{4}$ into (\ref{nonlinear2}),
we obtain (\ref{f2ug}).

Repeating the above process for $P_{j}(|u|^{2}u-|v|^{2}v)$ instead of
$P_{j}(|u|^{2}u)$ yields (\ref{f3ug}) and (\ref{f4ug}), we omit the details
and thus complete the proof of Proposition \ref{prop801}.
\end{proof}

\section{Proof of Theorem \ref{thm1}}

We will use iteration argument to prove Theorem \ref{thm1}.

 Let  $t\in [0, T]$. We define $u_{1}^{(0)}(t)$ and $u_{\infty}^{(0)}(t)$  as
 \begin{eqnarray}\label{801}
 \begin{cases}
   u_{1}^{(0)}(t)=\mbox{e}^{-tA}(1-\mbox{e}^{-TA})^{-1}\int_{0}^{T}\mbox{e}^{-(T-s)A}[P_{1}g]\mbox{d}s
+\int_{0}^{t}\mbox{e}^{-(t-s)A}[P_{1}g]\mbox{d}s, \\
   u_{\infty}^{(0)}(t)=\mbox{e}^{-tA}(1-\mbox{e}^{-TA})^{-1}\int_{0}^{T}\mbox{e}^{-(T-s)A}[P_{\infty}g]\mbox{d}s
+\int_{0}^{t}\mbox{e}^{-(t-s)A}[P_{\infty}g]\mbox{d}s.
\end{cases}
 \end{eqnarray}

We  inductively define $u_{1}^{(l)}(t)$ and $u_{\infty}^{(l)}(t)$ for $l\geq 1$ as
\begin{eqnarray}\label{802}
 \begin{cases}
   u_{1}^{(l)}(t)=\mbox{e}^{-tA}(1-\mbox{e}^{-TA})^{-1}
   \int_{0}^{T}\mbox{e}^{-(T-s)A}[F_{1}(u^{(l-1)}, g)]\mbox{d}s
+\int_{0}^{t}\mbox{e}^{-(t-s)A}[F_{1}(u^{(l-1)}, g)]\mbox{d}s, \\
   u_{\infty}^{(l)}(t)=\mbox{e}^{-tA}(1-\mbox{e}^{-TA})^{-1}
   \int_{0}^{T}\mbox{e}^{-(T-s)A}[F_{\infty}(u^{(l-1)}, g)]\mbox{d}s
+\int_{0}^{t}\mbox{e}^{-(t-s)A}[F_{\infty}(u^{(l-1)}, g)]\mbox{d}s,
\end{cases}
 \end{eqnarray}
where $u^{(l-1)}=u_{1}^{(l-1)}+u_{\infty}^{(l-1)}.$

It is easy to see that $ u_{1}^{(l)}(0)= u_{1}^{(l)}(T)$
and $ u_{\infty}^{(l)}(0)= u_{\infty}^{(l)}(T)$ for $l\geq0.$

We first show the existence of local solution to system (\ref{gl01})
satisfying
$u_{j}(0)=u_{j}(T)(j=0,1).$
\begin{prop}\label{prop804}
Let  $\delta_{2}>0$ be a constant small enough and let  $[g]\leq \delta_{2}.$
Then  there uniquely exists  solution
 $\{u_{1},u_{\infty}\}$ to system (\ref{gl01}) on $[0, T]$ in $B_{Z(0, T)}(C_{3}[g])$
 with $u_{j}(0)=u_{j}(T)(j=0,1)$.
 The uniqueness holds  in $B_{Z(0, T)}(C_{3}\delta_{2}).$
\end{prop}
\begin{proof}
From Proposition \ref{myprop1}, Proposition \ref{prop6050}, (\ref{f1ug}) and (\ref{f2ug}),
we deduce that there is two  constants $C_{1}, C_{2}>0$  independent of $g$ and $l$
such that there holds
\begin{equation}\label{u1l0}
  \|\{u_{1}^{(0)},u_{\infty}^{(0)}\}\|_{Z(0, T)}
  \leq
  C_{2}[g]
\end{equation}
and such that for all $l\geq0$ there holds
\begin{equation}\label{u1l}
  \|\{u_{1}^{(l+1)},u_{\infty}^{(l+l)}\}\|_{Z(0, T)}
  \leq C_{1}\|\{u_{1}^{(l)},u_{\infty}^{(l)}\}\|_{Z(0, T)}^{3}
  +C_{2}[g].
\end{equation}
Setting $l=0$ in (\ref{u1l}) and employing (\ref{u1l0}) leads to
\begin{equation}\label{u1l1}
  \|\{u_{1}^{(1)},u_{\infty}^{(l)}\}\|_{Z(0, T)}
  \leq C_{1}\|\{u_{1}^{(0)},u_{\infty}^{(0)}\}\|_{Z(0, T)}^{3}
  +C_{2}[g]
  \leq (C_{1}C_{2}^{3}[g]^{2}+C_{2})[g].
\end{equation}
Choosing $\delta_{1}<\sqrt{\frac{1}{C_{1}C_{2}^{3}}}$ yields
\begin{equation}\label{u1l10}
  \|\{u_{1}^{(1)},u_{\infty}^{(1)}\}\|_{Z(0, T)}
  \leq C_{1}\|\{u_{1}^{(0)},u_{\infty}^{(0)}\}\|_{Z(0, T)}^{3}
  +C_{2}[g]
  \leq (1+C_{2})[g]
\end{equation}
provided $[g]\leq \delta_{1}.$

By induction, one easily concludes
\begin{equation}\label{u1l12}
  \|\{u_{1}^{(l)},u_{\infty}^{(l)}\}\|_{Z(0, T)}
  \leq (1+C_{2})[g]
\end{equation}
for all $l\geq 2$ if choosing $\delta_{1}<\sqrt{\frac{1}{C_{1}(C_{2}+1)^{3}}}$.

Next, we easily see that $\bar{u}_{j}^{(l)}=u_{j}^{(l+1)}-u_{j}^{(l)}$ for $j=1,\infty$
satisfy the following equation
\begin{eqnarray*}
\begin{cases}
  \partial_{t}\bar{u}_{1}^{(l)}+A \bar{u}_{1}^{(l)}=F_{1}(u^{(l)},g)-F_{1}(u^{(l-1)},g), \\
  \partial_{t}\bar{u}_{\infty}^{(l)}+A \bar{u}_{\infty}^{(l)}=F_{\infty}(u^{(l)},g)-F_{\infty}(u^{(l-1)},g),
   \end{cases}
\end{eqnarray*}
where $u^{(l)}=u_{1}^{(l)}+u_{\infty}^{(l)}.$

Invoking Proposition \ref{myprop1}, Proposition \ref{prop6050},
 (\ref{f3ug}), (\ref{f4ug}) and (\ref{u1l12}), one can deduce that
\begin{equation}\label{u1l13}
  \|\{u_{1}^{(l+1)}-u_{1}^{(l)},u_{\infty}^{(l+1)}-u_{\infty}^{(l)}\}\|_{Z(0, T)}
  \leq C(1+C_{2})^{2}[g]^{2}\|\{u_{1}^{(l)}-u_{1}^{(l-1)},u_{\infty}^{(l)}-u_{\infty}^{(l-1)}\}\|_{Z(0, T)}
\end{equation}
for all $l\geq 1$ if $[g]\leq \delta_{1}.$

Setting $\delta_{2}=\min\{\delta_{1}, \frac{1}{2 C(1+C_{2})^{2}}\}$.
Let   $[g]\leq \delta_{2}$. Then $u_{j}^{(l)}(j=1, \infty)$
converges to $u_{j}(j=1, \infty)$ in the sense
\begin{eqnarray}
  \{u_{1}^{(l)}, u_{\infty}^{(l)}\}\rightarrow \{u_{1}, u_{\infty}\} \quad in \quad Z(0, T).\label{uinfty}
   %u_{\infty}^{(l)}\rightarrow u_{\infty} \text{*-weakly in $L^{\infty}(0, T;H_{1}^{2})$},\nonumber\\
%  u_{\infty}^{(l)}\rightarrow u_{\infty} \text{weakly in $ L^{2}(0, T;H_{1}^{3})\cap H^{1}(0, T;H_{1}^{1}) $}.
\end{eqnarray}
And one can easily verify that $\{u_{1}, u_{\infty}\}$ is a solution of (\ref{gl01}) satisfying
$u_{j}(0)=u_{j}(T)(j=1, \infty).$

%We are now in a position to verify $u_{\infty}\in C([0, T], H^{2}_{1}).$
%In fact,
% (\ref{uinfty}) and equation (\ref{gl01}) yields
% $F_{\infty}(u, g)\in L^{2}(0, T;H_{1}^{1}).$
% Combining this fact with the weighted energy method
% applied to equation (\ref{gl01}), one can easily deduce
% $u_{\infty}\in C([0, T], H^{2}_{1}).$
 We thus complete the proof of Proposition \ref{prop804}.
\end{proof}

\begin{prop}\label{prop805}
Assume that $s\in\mathbb{R}$
and $U_{0}=U_{01}+U_{0\infty},$ where
$\operatorname{supp}\hat{U}_{01}\subset  \{|\xi|\leq r_{\infty}\}$
, $\|U_{01}\|_{L^{2}}+\|x\nabla U_{01}\|_{L^{2}}<\infty$
and $U_{0\infty}\in X_{\infty}^{2}.$
If there is a sufficiently small constant $\delta_{3}$ such that
\begin{equation*}
  M(U_{01},U_{0\infty},g)\equiv
  \|U_{01}\|_{L^{2}}+\|x\nabla U_{01}\|_{L^{2}}
  +\|U_{0\infty}\|_{H_{1}^{2}}+[g]\leq \delta_{3}
\end{equation*}
then there is a constant $C_{4}$ such that there uniquely exists a solution $\{U_{1}, U_{\infty}\}$
to  system (\ref{gl01}) on $[s,s+T]$ in $B_{Z(s, s+T)}(C_{4}M(U_{01},U_{0\infty},g))$
satisfying $U_{j}|_{t=0}=U_{0j}(j=1, \infty).$
The uniqueness holds in $B_{Z(s, s+T)}(C_{4}\delta_{3}).$
\end{prop}

\begin{proof}
Using iteration argument, the proof of Proposition \ref{prop805}
is similar to the proof of Proposition \ref{prop804} and we omit the details.
\end{proof}
\textbf{Proof of Theorem \ref{thm1}:}
Having Propositions \ref{prop804} and  \ref{prop805} in hand, we
can use similar method as that in \cite{KageiTsuda2015} to prove
Theorem \ref{thm1}, we present the details here for the reader's convenience.

First, according to Proposition \ref{prop804}, one derives that if $[g]\leq \delta_{2},$
then equation (\ref{gl01}) admits a unique solution
$\{u_{1}^{(0)},u_{\infty}^{(0)}\}\in B_{Z(0, T)}(C_{3}[g])$
satisfying $u_{1}^{(0)}(0)=u_{1}^{(0)}(T)$ and $u_{\infty}^{(0)}(0)=u_{\infty}^{(0)}(T)$.
This implies that
\begin{equation}\label{0810}
  \sup_{t\in [0,T]}\left(
  \|u_{1}^{(0)}\|_{H^{1}_{1}}+\|u_{\infty}^{(0)}\|_{H^{2}_{1}}
  \right)\leq C_{3}[g].
\end{equation}
Let $[g]\leq \frac{\delta_{3}}{ C_{3}+1}$, then  Proposition \ref{prop805}
allows us to conclude that equation (\ref{gl01}) admits a unique solution
$\{u_{1}^{(1)},u_{\infty}^{(1)}\}\in B_{Z(T, 2T)}(C_{4}(C_{3}+1)[g])$
satisfying $u_{1}^{(1)}(T)=u_{1}^{(0)}(T)=u_{1}^{(0)}(0)$
and $u_{\infty}^{(1)}(T)=u_{\infty}^{(0)}(T)=u_{\infty}^{(0)}(0).$

Next, we define $v_{j}^{(1)}(j=1,\infty)$ and $v^{(1)}$  for  $t\in[0,T]$  by
\begin{equation*}
  v_{j}^{(1)}(t)=u_{j}^{(1)}(t+T),\quad v^{(1)}(t)=v_{1}^{(1)}(t)+v_{\infty}^{(1)}(t).
\end{equation*}
The it is easy to see that equation (\ref{gl01}) is satisfied with
$v_{j}^{(1)}(j=1,\infty)$ instead of  $u_{j}$ and $v^{(1)}$ instead of $u.$
And the initial data are $v_{j}^{(1)}(0)=u_{j}^{(0)}(0).$

Let $[g]\leq \min\{\delta_{2}, \frac{C_{4}\delta_{3}}{C_{3}}, \frac{\delta_{3}}{C_{3}+1}\}$.
Then $ v_{j}^{(1)}(t)=u_{j}^{(0)}(t)(j=1,\infty)$ for $t\in[0,T]$ due to uniqueness.
This yields $u_{j}^{(1)}(t)=u_{j}^{(0)}(t-T)$ for $t\in [T, 2T].$

For $j=1, \infty$, define
\begin{eqnarray*}
  u_{j}(t)
  =\begin{cases}
  u_{j}^{(0)}(t), \text{$t\in [0,T]$,} \\
  u_{j}^{(1)}(t), \text{$t\in [T,2T]$.}
  \end{cases}
\end{eqnarray*}
Then for $t\in [0,T],$ one finds $u_{j}(t+T)=u_{j}(t).$
Proposition \ref{prop805} and (\ref{0810})
allow us to conclude that
there uniquely exists a solution
$\{w_{1},w_{\infty}\}\in B_{Z(T/2, 3T/2)}(C_{4}(C_{3}+1)[g])$
to equation (\ref{gl01}) on $[T/2, 3T/2]$
satisfying $w_{j}(T/2)=u_{j}^{(0)}(T/2)$
and uniqueness yields $w_{j}=u_{j}$ on $[T/2, 3T/2]$.
This shows that $\{u_{1}, u_{\infty}\}$ satisfies equation (\ref{gl01})
in $Z(0, 2T).$
We then can obtain solution $\{u_{1}, u_{\infty}\}$ to (\ref{gl01})
in $Z_{per}(\mathbb{R})$ satisfying
 $\|\{u_{1}, u_{\infty}\}\|_{Z(0, T)}\leq C_{3}[g]$
by repeating this argument on $[mT, (m+1)T]$ for $m=\pm1, \pm2,\cdots.$
The periodic solution of (\ref{gl}) is then obtained by setting $u=u_{1}+u_{\infty}.$
The proof of the uniqueness part is due to the iteration argument, we omit the details.
We thus complete the proof of Theorem \ref{thm1}.

\section{Proof of Theorem \ref{thm2}}

The existence of the unique global solution stated in
Theorem \ref{thm2}
can be proved
using similar method as that in \cite{Kawashita2002},
we omit the details and just show the decay rates.
The method is due to \cite{Okita2014}.

Applying $P_{1}$ and $P_{\infty}$ to (\ref{gldiff}), respectively, leads to
\begin{equation}\label{lowdiff}
  \partial_{t}w_{1}-(1+i)\Delta w_{1}
  =P_{1}(2|v_{per}|^{2}w+v_{per}^{2}\bar{w}+|w|^{2}w+|w|^{2}v_{per})
\end{equation}
and
\begin{equation}\label{highdiff}
   \partial_{t}w_{\infty}-(1+i)\Delta w_{\infty}
  =P_{\infty}(2|v_{per}|^{2}w+v_{per}^{2}\bar{w}+|w|^{2}w+|w|^{2}v_{per}).
\end{equation}

Define
\begin{eqnarray*}
  &&N_{1}(t) \equiv \sup_{0\leq \tau\leq t}
  [(1+\tau)^{3/4}\|w_{1}(\tau)\|_{L^{2}}+(1+\tau)^{5/4}\|\nabla w_{1}(\tau)\|_{L^{2}}],\\
  &&N_{2}(t) \equiv \sup_{0\leq \tau\leq t}
  (1+\tau)^{5/4}\| w_{\infty}(\tau)\|_{H^{1}}, \\
  &&N(t)\equiv N_{1}(t)+ N_{2}(t).
\end{eqnarray*}

To prove Theorem \ref{thm2}, we first show the following two
Lemmas.

For the low frequency part, we will show
\begin{lem}\label{lemstability1}
There exists a $\epsilon>0$ such that if
$[g]\leq \epsilon$ and $N(t)\leq 1$ for $t\in[0, T]$.
Then there is a constant $C_{5}$ independent of $T$ such that
\begin{equation}\label{lowinequality}
  N_{1}(t)\leq C\|w_{0}\|_{L^{1}}
  +C\epsilon N(t)+CN^{2}(t)
\end{equation}
for $t\in[0,T]$.
\end{lem}

\begin{proof}
The solution of (\ref{lowdiff}) can be written as
\begin{equation}\label{stable0}
  w_{1}(t)=\mbox{e}^{-(1+i)t\Delta}w_{10}
  +\int_{0}^{t}\mbox{e}^{-(1+i)(t-s)\Delta}
  P_{1}(2|v_{per}|^{2}w+v_{per}^{2}\bar{w}+|w|^{2}w+|w|^{2}v_{per})(s)\mbox{d}s,
\end{equation}
where $w_{10}=P_{1}w_{0}.$

Invoking Plancherel theorem yields
\begin{eqnarray}\label{stable1}
  &&\|\nabla^{l}P_{1}w_{0}\|_{L^{2}}\nonumber\\
  &&\leq
  \left[\int_{|\xi|\leq r_{\infty}}|\xi|^{2l}\mbox{e}^{-|\xi|^{2}}|\hat{w}_{0}(\xi)|^{2}\mbox{d}\xi\right]^{1/2}\nonumber\\
  &&\leq
  \min\left\{
  \left[\int_{|\xi|\leq r_{\infty}}|\xi|^{2l}\mbox{e}^{-t|\xi|^{2}}|\hat{w}_{0}(\xi)|^{2}\mbox{d}\xi\right]^{1/2},
  \left[\int_{|\xi|\leq 1}\mbox{e}^{-t|\xi|^{2}}|\hat{w}_{0}(\xi)|^{2}\mbox{d}\xi\right]^{1/2}
  \right\}\nonumber\\
  &&\leq
  \min\left\{C\|\hat{w}_{0}\|_{L^{\infty}}
  \left[\int_{\mathbb{R}^{3}}|\xi|^{2l}\mbox{e}^{-t|\xi|^{2}}\mbox{d}\xi\right]^{1/2},
  C\|\hat{w}_{0}\|_{L^{\infty}}\left[\int_{|\xi|\leq 1}\mbox{e}^{-t|\xi|^{2}}\mbox{d}\xi\right]^{1/2}
  \right\}\nonumber\\
  &&\leq
  \min\left\{C\|\hat{w}_{0}\|_{L^{\infty}}
  \left[\int_{\mathbb{R}^{3}}t^{-l}|\theta|^{2l}\mbox{e}^{-|\theta|^{2}}t^{-3/2}\mbox{d}\theta\right]^{1/2},
  C\|\hat{w}_{0}\|_{L^{\infty}}
  \right\}\nonumber\\
  &&\leq C(1+t)^{-(3/4+l/2)}\|w_{0}\|_{L^{1}}.
\end{eqnarray}
Using Lemma \ref{lemhardy}, H\"{o}lder inequality and Theorem \ref{thm1}, it finds
\begin{eqnarray}\label{stable2}
&&\|2|v_{per}|^{2}w+v_{per}^{2}\bar{w}+|w|^{2}w+|w|^{2}v_{per}\|_{L^{1}}\nonumber\\
&&\leq C\|v_{per}\|_{L^{\infty}}
\|(1+|x|)v_{per}\|_{L^{2}}\|\frac{1}{(1+|x|)}w\|_{L^{2}}
+\|w\|_{L^{3}}^{3}
+\|v_{per}\|_{L^{3}}\|w\|_{L^{2}}\|w\|_{L^{6}}\nonumber\\
&&\leq
C\epsilon^{2}(\|\nabla w_{1}\|_{L^{2}}+\|\nabla w_{\infty}\|_{L^{2}})
+C(\|w_{1}\|_{L^{2}}+\|w_{\infty}\|_{L^{2}})^{3/2}(\|\nabla w_{1}+\nabla w_{\infty}\|_{L^{2}})^{3/2}\nonumber\\
&&\qquad+C\epsilon(\|w_{1}\|_{L^{2}}+\|w_{\infty}\|_{L^{2}})(\|\nabla w_{1}+\nabla w_{\infty}\|_{L^{2}})\nonumber\\
&&\leq C\epsilon(1+t)^{-5/4}N(t)
+C(1+t)^{-3}N^{3}(t)
+C\epsilon(1+t)^{-2}N^{2}(t).
\end{eqnarray}
Applying $\nabla^{l}$ to (\ref{stable0}) and employing (\ref{stable1})(\ref{stable2}) leads to
\begin{eqnarray}\label{stable3}
&&\|\nabla^{l}w_{1}(\tau)\|_{L^{2}}\nonumber\\
&&\leq C(1+\tau)^{-(3/4+l/2)}\|w_{0}\|_{L^{1}}\nonumber\\
&&\quad+C\int_{0}^{t}(1+\tau-s)^{-(3/4+l/2)}
\|2|v_{per}|^{2}w+v_{per}^{2}\bar{w}+|w|^{2}w+|w|^{2}v_{per}\|_{L^{1}}\mbox{d}s\nonumber\\
&&\leq C(1+\tau)^{-(3/4+l/2)}\|w_{0}\|_{L^{1}}\nonumber\\
&&\quad+C\int_{0}^{\tau}(1+\tau-s)^{-(3/4+l/2)}[\epsilon(1+s)^{-5/4}N(s)\nonumber\\
&&\quad\quad\quad\quad+(1+s)^{-3}N^{3}(s)
+\epsilon(1+s)^{-2}N^{2}(s)]\mbox{d}s\nonumber\\
&&\leq C(1+\tau)^{-(3/4+l/2)}\|w_{0}\|_{L^{1}}\nonumber\\
&&\quad+C\epsilon(1+\tau)^{-(3/4+l/2)}N(t)+C(1+\tau)^{-(3/4+l/2)}N^{2}(t)
\end{eqnarray}
for $\tau\in [0,t].$ This yields (\ref{lowinequality})
and we complete the proof of Lemma \ref{lemstability1}.
\end{proof}

For the high frequency part, we will show

\begin{lem}\label{lemstability2}
There exists a $\epsilon>0$ such that if
$[g]\leq \epsilon$ and $N(t)\leq 1$ for $t\in[0, T]$.
Then we have
\begin{eqnarray}\label{highinequality}
  &&\frac{1}{2}\frac{\mbox{d}}{\mbox{d}t}\|w_{\infty}\|_{H^{1}}^{2}
  +(d_{3}\|  w_{\infty}\|_{H^{1}}^{2})
  +(d_{2}\| \nabla w_{\infty}\|_{H^{1}}^{2})\nonumber\\
  &&\leq C\epsilon(1+t)^{-5/2}N^{2}(t)
 +C(1+t)^{-5}N^{4}(t)
  +C(1+t)^{-5/2}N^{2}(t)\|\nabla w_{\infty}\|_{H^{1}}^{2}.
\end{eqnarray}
for $t\in[0,T]$ and some constant $d_{2}, d_{3}$ large enough.
\end{lem}

\begin{proof}
The standard energy method yields
\begin{eqnarray}\label{stable4}
  &&\frac{1}{2}\frac{\mbox{d}}{\mbox{d}t}\|w_{\infty}\|_{H^{1}}^{2}+
  \| \nabla w_{\infty}\|_{H^{1}}^{2}\nonumber\\
  &&\leq \|2|v_{per}|^{2}w+v_{per}^{2}\bar{w}+|w|^{2}w+|w|^{2}v_{per}\|_{L^{2}}\|\nabla^{2}w_{\infty}\|_{L^{2}},
\end{eqnarray}
where integration by parts and Lemma \ref{lem404}(ii) have been employed.

Direct calculation leads to
\begin{eqnarray}\label{stable5}
&&\|2|v_{per}|^{2}w+v_{per}^{2}\bar{w}+|w|^{2}w+|w|^{2}v_{per}\|_{L^{2}}\nonumber\\
&&\leq C\|v_{per}\|_{L^{6}}^{2}\|w\|_{L^{6}}
+C\|w\|_{L^{6}}^{3}+C\|v_{per}\|_{L^{6}}\|w\|_{L^{6}}^{2}\nonumber\\
&&\leq C\epsilon^{2}(\|\nabla w_{1}\|_{L^{2}}+\|\nabla w_{\infty}\|_{L^{2}})
+C(\|\nabla w_{1}\|_{L^{2}}+\|\nabla w_{\infty}\|_{L^{2}})^{3}\nonumber\\
&&\quad\quad+C\epsilon(\|\nabla w_{1}\|_{L^{2}}+\|\nabla w_{\infty}\|_{L^{2}})^{2}\nonumber\\
&&\leq C\epsilon(1+t)^{-5/4}N(t)
+C(1+t)^{-15/4}N^{3}(t)
+C\epsilon(1+t)^{-5/2}N^{2}(t).
\end{eqnarray}
Substituting (\ref{stable5}) into (\ref{stable4}) and using Young inequality, one obtains
\begin{eqnarray}\label{stable6}
  &&\frac{1}{2}\frac{\mbox{d}}{\mbox{d}t}\|w_{\infty}\|_{H^{1}}^{2}+
  \| \nabla w_{\infty}\|_{H^{1}}^{2}\nonumber\\
  &&\leq C\epsilon[(1+t)^{-5/2}N^{2}(t)+\|\nabla^{2}w_{\infty}\|_{L^{2}}^{2}]\nonumber\\
  &&\quad+C(1+t)^{-5/2}N^{2}(t)[(1+t)^{-5/2}N^{2}(t)+\|\nabla^{2}w_{\infty}\|_{L^{2}}^{2}]\nonumber\\
  &&\quad+C\epsilon[(1+t)^{-5}N^{4}(t)+\|\nabla^{2}w_{\infty}\|_{L^{2}}^{2}].
\end{eqnarray}
(\ref{stable6}) and Lemma \ref{lem404}(ii) yield (\ref{highinequality}).
%there is a constant $d_{2}$ large enough such that
%\begin{eqnarray}\label{stable6}
%% \nonumber to remove numbering (before each equation)
%  &&\frac{1}{2}\frac{\mbox{d}}{\mbox{d}t}(d_{2}\|w_{\infty}\|_{H^{2}}^{2})+
%  (d_{2}\|  w_{\infty}\|_{H^{1}}^{2})\nonumber\\
%  &&\leq C\delta_{0}(1+t)^{-5/2}N^{2}(t)
% +C(1+t)^{-5}N^{4}(t)
%  +C(1+t)^{-5/2}N^{2}(t)\|\nabla^{2}w_{\infty}\|_{L^{2}}^{2}.\nonumber
%\end{eqnarray}
We thus complete the proof of Lemma \ref{lemstability2}.
\end{proof}

\textbf{Proof of Theorem \ref{thm2}:}

We will show that
if there is a constant $\epsilon_{2}>0$
small enough such that $\|w_{0}\|_{H^{1}\cap L^{1}}\leq \epsilon_{2},$
then $N(t)\leq C\|w_{0}\|_{H^{1}\cap L^{1}}$ for $t\in[0, T]$ with $c$ independent of $T.$
A bootstrap argument yields  $N(t)\leq D_{0}$ for all $t$ with $D_{0}$ a constant
independent of $t.$

First, (\ref{highinequality}) implies
\begin{eqnarray}\label{stable7}
  &&\|w_{\infty}\|_{H^{1}}^{2}
+ d_{2}\int_{0}^{t}\mbox{e}^{-d_{3}(t-\tau)}\|\nabla w_{\infty}\|_{H^{1}}^{2}\mbox{d}\tau\nonumber\\
&&\leq \mbox{e}^{-d_{3}t}\|w_{\infty0}\|_{H^{1}}^{2}
+C\epsilon N^{2}(t)\int_{0}^{t}\mbox{e}^{-d_{3}(t-\tau)}(1+\tau)^{-5/2}\mbox{d}\tau\nonumber\\
&&\quad\quad+CN^{4}(t)\int_{0}^{t}\mbox{e}^{-d_{3}(t-\tau)}(1+\tau)^{-5}\mbox{d}\tau\nonumber\\
&&\quad\quad+CN^{2}(t)\int_{0}^{t}\mbox{e}^{-d_{3}(t-\tau)}(1+\tau)^{-5/2}\|\nabla w_{\infty}\|_{H^{1}}^{2}(\tau)\mbox{d}\tau\nonumber\\
&&\leq \mbox{e}^{-d_{3}t}\|w_{\infty0}\|_{H^{1}}^{2}
+C\epsilon (1+t)^{-5/2}N^{2}(t)
+C(1+t)^{-5}N^{4}(t)\nonumber\\
&&\quad\quad+CN^{2}(t)\int_{0}^{t}\mbox{e}^{-d_{3}(t-\tau)}\|\nabla w_{\infty}\|_{H^{1}}^{2}(\tau)\mbox{d}\tau.
\end{eqnarray}
Setting $\mathcal{A}(t)\equiv (1+t)^{5/2}\int_{0}^{t}\mbox{e}^{-d_{3}(t-\tau)}\|\nabla w_{\infty}\|_{H^{1}}^{2}(\tau)\mbox{d}\tau.$
Then one has from (\ref{stable7}) that
\begin{eqnarray}\label{stable8}
  N_{2}^{2}(t)+d_{2}\mathcal{A}(t)
  \leq C[\|w_{\infty0}\|_{H^{1}}^{2}+\epsilon N^{2}(t)+N^{4}(t)+N^{2}(t)\mathcal{A}(t)].
\end{eqnarray}

Adding (\ref{stable8}) to (\ref{lowinequality}) and choosing $\epsilon$
sufficiently small, we obtain
\begin{eqnarray}\label{stable9}
  N^{2}(t)+d_{2}^{\prime}\mathcal{A}(t)
  \leq C_{5}[\|w_{0}\|_{H^{1}\cap L^{1}}^{2}+N^{4}(t)+N^{2}(t)\mathcal{A}(t)].
\end{eqnarray}

On the other hand, we note that there is a constant
$ C_{6}>0$ such that
\begin{equation*}
  N(0)\leq C_{6} \|w_{0}\|_{H^{1}\cap L^{1}}.
\end{equation*}
The continuity of $N(t)$ in $t$ yields that there is a $t_{0}>0$
such that
\begin{equation*}
  N(t)\leq 2C_{6} \|w_{0}\|_{H^{1}\cap L^{1}} \quad \text{for all $t\in[0,t_{0}]$.}
\end{equation*}
Set $C_{7}=\max\{\sqrt{C_{5}/2}, C_{6}\}.$

We will prove that $N(t)< 2C_{7}\|w_{0}\|_{H^{1}\cap L^{1}}$
for all $t\in[0,T]$ if $\epsilon_{2}$ is selected suitably small.

Suppose this is not the case.
Then there is a $t_{1}\in (t_{0},T]$ such that
\begin{equation*}
  N(t)<2C_{7}\|w_{0}\|_{H^{1}\cap L^{1}} \quad \text{for all $t\in[0,t_{1})$}
\end{equation*}
and
\begin{equation}\label{nt1}
  N(t_{1})=2C_{7}\|w_{0}\|_{H^{1}\cap L^{1}}.
\end{equation}

However, from (\ref{stable9}), we deduce
\begin{eqnarray*}
  N^{2}(t)+d_{2}^{\prime}\mathcal{A}(t)
 &\leq& C_{5}\|w_{0}\|_{H^{1}\cap L^{1}}^{2}
  +C_{5}N^{2}(t)[N^{2}(t)
  +\mathcal{A}(t)]\\
  &<& C_{5}\|w_{0}\|_{H^{1}\cap L^{1}}^{2}
  +\frac{1}{2}[N^{2}(t)
  +d_{2}^{\prime}\mathcal{A}(t)]\quad \text{for all $t\in[0,t_{1}]$}
\end{eqnarray*}
provided we choose
$\epsilon_{2}<\epsilon_{3}\equiv\min\left\{\frac{1}{\sqrt{8C_{5}C_{7}^{2}}}, \frac{\sqrt{d_{2}^{\prime}}}{\sqrt{8C_{5}C_{7}^{2}}}\right\}$.

The above inequality leads to
\begin{eqnarray*}
  N^{2}(t)+d_{2}^{\prime}\mathcal{A}(t)
  < 2C_{5}\|w_{0}\|_{H^{1}\cap L^{1}}^{2}\leq 4C_{7}^{2}\|w_{0}\|_{H^{1}\cap L^{1}}^{2} \quad \text{for all $t\in[0,t_{1}]$,}
\end{eqnarray*}
which contradicts with (\ref{nt1}).

We thus complete the proof of Theorem \ref{thm2}.

\textbf{Acknowledgement}

The research of B. Guo
is partially supported by the National Natural Science Foundation
of China, grant 11731014.

%%%%%%%%%%%%%%%%%%%%%%%%%%%%%%%%%%%%%%%%%%%%%%%%%%%%%%%%%%%% References %%%%%%%%%%%%%%%%%%%%%%%%%%%%%%%%%%%%%%%%%%%%%%%%%%%%%%%%% \begin{thebibliography}{}

\end{document}